\documentclass[12pt]{article}
\usepackage{amsmath,amsxtra,amssymb,color,latexsym,epsfig,amscd,amsthm,subfigure,fancybox,epsfig}
\usepackage[mathscr]{eucal}
\usepackage{graphicx}
\usepackage{enumerate}
\usepackage{epsfig} 
\usepackage{epstopdf}
\usepackage{cases}
\newtheorem{thm}{Theorem}[section]
\newtheorem{prop}[thm]{Proposition}
\newtheorem{lm}[thm]{Lemma}
\newtheorem{rem}[thm]{Remark}

\newtheorem{exm}[thm]{Example}
\newtheorem{asp}{Assumption}[section]
\numberwithin{equation}{section}

\DeclareMathOperator{\trace}{tr}
\newcommand{\eps}{\varepsilon}

\newcommand{\h}{\mathcal{H}}
\newcommand{\A}{\mathcal{A}}
\newcommand{\B}{\mathcal{B}}
\newcommand{\p}{\mathfrak{p}}
\newcommand{\m}{\mathfrak{m}}
\newcommand{\C}{\mathcal{C}}

\newcommand{\F}{\mathcal{F}}

\newcommand{\E}{\mathbb{E}}
\newcommand{\LL}{\mathcal{L}}
\newcommand{\N}{{\mathbb{Z}}_+}
\newcommand{\Z}{\mathbb{Z}}
\newcommand{\K}{\mathcal{K}}
\newcommand{\PP}{\mathbb{P}}

\newcommand{\R}{\mathbb{R}}

\numberwithin{equation}{section}
\newcommand{\1}{\boldsymbol{1}}

\newcommand{\wdt}{\widetilde}

\newcommand{\bed}{\begin{displaymath}}
\newcommand{\eed}{\end{displaymath}}
\newcommand{\bea}{\bed\begin{array}{rl}}
\newcommand{\eea}{\end{array}\eed}

\newcommand{\ad}{&\!\!\!\disp}

\newcommand{\barray}{\begin{array}{ll}}
\newcommand{\earray}{\end{array}}

\def\disp{\displaystyle}
\newcommand{\rr}{{\Bbb R}}
\newcommand{\al}{\alpha}
\newcommand{\cd}{(\cdot)}

\def\bar{\overline}
\def\hat{\widehat}
\def\a.s{\text{\;a.s.\;}}

\begin{document}
\title{Modeling and Analysis of Switching Diffusion Systems:
Past-Dependent Switching with a Countable State Space\thanks{This
research was supported in part by the National Science Foundation under grant DMS-1207667.}}

\author{Dang Hai Nguyen\thanks{Department of Mathematics, Wayne State University, Detroit, MI
48202,
dangnh.maths@gmail.com.} \and
George Yin\thanks{Department of Mathematics, Wayne State University, Detroit, MI
48202,
gyin@math.wayne.edu.}}
\maketitle

\begin{abstract} Motivated by networked systems in random environment and controlled hybrid stochastic dynamic systems,
this work focuses on modeling and analysis of a class of
switching diffusions consisting of continuous and discrete components.
Novel features of the models include the discrete
component taking values in a countably infinite set, and the switching depending on the value of the
continuous component involving past history.
In this work, the existence and uniqueness of solutions of the associated stochastic differential equations are obtained. In addition,
Markov and Feller properties of a function-valued stochastic process associated with the hybrid diffusion are also proved.
In particular, when the switching rates depend only on the current state, strong Feller properties are obtained. These properties will pave
a way for future study of control design and optimization of such dynamic systems.

\bigskip
\noindent {\bf Keywords.} Switching diffusion, past-dependent switching,
countable state space, existence and uniqueness of solution, Feller property.

\bigskip
\noindent{\bf Mathematics Subject Classification.} 93E03, 60J60, 60H10, 92D25.

\bigskip
\noindent{\bf Running Title.}
Past-Dependent-Switching Diffusion Systems

\end{abstract}

\newpage

\section{Introduction}\label{sec:int}
Owing to the demand of modeling, analysis, and computation of complex networked systems, much attention has been devoted to building more realistic dynamic system models.
It has been well recognized that
in  many real-world applications,
 traditional models  using continuous processes represented by
 solutions to deterministic differential equations and stochastic differential equations alone  are often inadequate.
Arising from control engineering,
queueing networks,
manufacturing and production planning, parameter estimation,
 filtering  of dynamic systems,
 ecological and biological systems,
and financial engineering, etc., numerous complex systems contain both continuous dynamics and discrete events.
The discrete events in these systems are not normally representable by solutions of the usual differential equations.
Because of the demand, switching diffusions (also known as hybrid switching diffusions) have drawn growing and resurgent attention.
A switching diffusion is a two-component process $(X(t),\alpha(t))$ in which the continuous component $X(t)$ evolves according to the diffusion process whose drift and diffusion coefficients depend on the state of $\alpha(t)$, whereas $\alpha(t)$ takes values in a set consisting of isolated points.
Because of their importance, many papers have been devoted to such hybrid dynamic systems; see \cite{KZY,SX,YZ2,ZY}
and the references therein.
In their comprehensive treatment of hybrid switching diffusions,
Mao and Yuan \cite{MY} focused on $\alpha(t)$ being a continuous-time and homogeneous Markov chain  independent of the Brownian motion
and the generator of the Markov chain being a constant matrix.
 Realizing the need, treating the two components jointly,
 Yin and Zhu \cite{YZ} extended the study to the
 Markov process  $(X(t),\alpha(t))$
by allowing the generator $\alpha(t)$ to depend on the current state $X(t)$.
Properties of the underlying process including recurrence, positive recurrence, ergodicity, Feller properties, stability,
and invariance among others were investigated. Such study  provides us with a clear picture of the underlying processes. Nevertheless,
in both of the aforementioned books and most related papers
to date,
the switching process $\alpha(t)$ is assumed to have a finite state space.
One question naturally arises. What happens if the switching process has a countable state space? Much of the argument in \cite{YZ} relies on the interplay of stochastic processes and the associated systems of partial differential equations. Because the state space of $\alpha(t)$ was assumed to be a finite set,
one can essentially treat a system of partial differential equations with a finite number of equations. When we consider problems  involving
 a countable state space, the number of equations becomes infinite. Much more complex situation is encountered. Different methods have to be
 developed to treat the systems.

There are plenty of real-world applications involving such switching diffusions.
Perhaps, one of the most widely used control models in the literature is the so-called LQG (linear quadratic Gaussian regulator) problem;
see \cite[pp.165-166]{FlemingR} for a
traditional model. However, for many new applications in networked systems, it has been found that in addition to the random noise represented by Brownian type of disturbances, there is a source of randomness owing to the presence of random environment
that can be modeled by a continuous-time Markov chain.
Let $\al(t)$ be a continuous-time Markov chain with state space
$\N$ (the set of positive integers) and generator $Q$.
Consider the controlled dynamic system
\begin{equation}\label{system-eqn}
\begin{array}{ll} &\disp
d X(t)=[A(\al(t))X(t)+B(\alpha(t))u(t)] dt + \sigma(\al(t)) dW(t) ,\\
& \disp X(s)=x, \hbox{ for } s\le t\le T, \end{array}\end{equation}
where $X(t)\in \rr^{n_1}$ is the continuous state variable,
$u(t)\in \rr^{n_2}$ is the control,
$A(i)\in \rr^{n_1\times n_1}$ and
$B(i) \in \rr^{n_1\times n_2}$ are well defined and have finite
values for each $i\in \N$.
One may wish to find
the optimal control $u\cd$ so that the
expected quadratic cost function
\begin{equation}\label{cost-fn}
\disp
J(s,i,x,u(\cdot))= E\Big[\!\!\int^T_s [X^\top (t) M(\al(t))X(t)
+u^\top(t) N(\al(t)) u(t)] dt+X^\top(T)D X(T) \Big]
\end{equation}
is minimized.
The use of $\al(t)$ stems from the formulation of discrete events, and  the use of $\N$ enlarges the applicability of previous consideration of
finite state space cases.
Switched dynamic systems
 can also be found in, for example,  modeling
 impatient customers and
customer abandonment of
Markov-modulated service speeds in the  heavy-traffic regime and the many-server
systems in the Halfin-Whitt regime and the non-degenerate slowdown regime; see \cite{HZZ}.
 We also refer the reader to
Whitt
\cite{WW} for further reading on limit results  in queueing theory and many references therein.
In fact, in most of the queueing models, the discrete set is countable rather than finite.

Two more dynamic systems are to be presented in the next section,
in which the main interests are to find long-term behavior and control design
 in an ecological system and to find optimal strategies under long-run average criteria for  a pollution management problem.
 In order to study  the aforesaid problems, we first need to ensure that the systems under consideration have unique solutions
 and that the solutions possess good properties.
Motivated by these examples, we take up the challenge of considering a nonlinear hybrid diffusion $(X(t),\alpha(t))$
whose discrete component $\alpha(t)$ has an infinite state space in this paper.
Moreover, in lieu of allowing the switching process to depend on the current state $X(t)$ only, we assume that it is past dependent.
That is,
we assume that the generator of $\alpha(t)$ depends on
the past history of the continuous process.
This paper provides conditions for the existence and uniqueness
  of the solutions for given initial data, and to
  demonstrate the Markov-Feller property of 
  a function-valued stochastic process associated with the equation. Our study will build a bridge for future study on related control systems.

The rest of the paper is organized as follows. The formulation of hybrid switching diffusions with past-dependent switching and
countably many possible switching locations is given in Section \ref{sec:for}.
The existence and uniqueness of solutions to
the stochastic equations are then proved under suitable conditions in Section \ref{sec:exi}.
Section \ref{sec:mar} studies the Markov and Feller properties of a function-valued stochastic process associated with our equation.
The proof for the Feller property is rather complex because the state space of $\alpha(t)$ is infinite, the  space of continuous functions is not locally compact, and we do not assume uniform continuity of the
switching intensities.
In section \ref{sec:fel}, the strong Feller property of the hybrid diffusion without past-dependent switching is given.
Section \ref{sec:rem} provides further remarks and points out future research directions.
Finally, we provide the proofs of some technical results in an appendix.

\section{Formulation}\label{sec:for}
Let $r$ be a fixed positive number.
Denote by $\C([a,b],\R^n)$ the set of  $\R^n$-valued continuous functions defined on $[a, b]$.
In what follows,
we mainly work with $\C([-r,0],\R^n)$, and simply denote it by $\C:=\C([-r,0],\R^n)$.
Denote by $|x|$ the Euclidean norm of $x\in\R^n$.
For $\phi\in\C$, we use the norm $\|\phi\|=\sup\{|\phi(t)|: t\in[-r,0]\}$.
For $y(\cdot)\in\C([-r,\infty),\R^n)$ and $t\geq0$, we denote  by $y_t$
the
so-called segment function (or memory segment function) $y_t(\cdot):= y(t+\cdot)\in\C$.
Let $(\Omega,\F,\{\F_t\}_{t\geq0},\PP)$ be a complete filtered probability space with the filtration $\{\F_t\}_{t\geq 0}$ satisfying the usual condition,  i.e., it is increasing and right continuous while $\F_0$ contains all $\PP$-null sets.
Let $W(t)$ be an $\F_t$-adapted and $\R^d$-valued Brownian motion.
Suppose $b(\cdot,\cdot): \R^n\times\N\to\R^n$ and $\sigma(\cdot,\cdot): \R^n\times\N\to\R^{n\times d}$, where $\N= {\mathbb N} \setminus \{0\}=\{1,2,\dots\}$,  the set of positive integers.
Consider the two-component process $(X(t),\alpha(t))$, where
 $\alpha(t)$ is a pure jump process taking value in  $\N$, and $X(t)$ satisfies
\begin{equation}\label{eq:sde} dX(t)=b(X(t), \alpha(t))dt+\sigma(X(t),\alpha(t))dW(t).\end{equation}
{
We assume that if $\alpha(t-):=\lim_{s\to t^-}\alpha(s)=i$, then it can switch to $j$ at $t$ with intensity $q_{ij}(X_t)$ where $q_{ij}(\cdot):\C\to\R$.
When $q_{i}(\phi):=\sum_{j=1,j\ne i}^\infty q_{ij}(\phi)$ is uniformly bounded in $(\phi,i)\in\C\times\N$, and
$q_{i}(\cdot)$ and $q_{ij}(\cdot)$ are continuous,
one may view
the aforementioned assumption as
\begin{equation}\label{eq:tran}\begin{array}{ll}
&\disp \PP\{\alpha(t+\Delta)=j|\alpha(t)=i, X_s,\alpha(s), s\leq t\}=q_{ij}(X_t)\Delta+o(\Delta) \text{ if } i\ne j \
\hbox{ and }\\
&\disp \PP\{\alpha(t+\Delta)=i|\alpha(t)=i, X_s,\alpha(s), s\leq t\}=1-q_{i}(X_t)\Delta+o(\Delta).\end{array}\end{equation}
However, when $q_i(\phi)$ and  $q_{ij}(\phi)$ are either discontinuous or unbounded, it does not seem appropriate to use
\eqref{eq:tran} to model the switching intensity.
To formulate the problem in a general setting without the boundedness and continuity assumptions mentioned above,
we construct $\alpha(t)$ as the solution to a stochastic differential equation with respect to a Poisson random measure.
We elaborate on the idea below.
Let $\p(dt , dz)$ be a Poisson random measure with intensity $dt\times\m(dz)$ and $\m$ is the Lebesgue measure on $\R$ such that
$\p(\cdot, \cdot)$ is
independent of the Brownian motion $W(\cdot)$.
Let $\tilde\p$ be the Poisson point process associated with $\p(\cdot,\cdot)$ (see e.g., \cite{SR}).
Then $\tilde\p$ can lie in a set $A$ with intensity $\m(A)$, that is,
the expected number of Poisson points lying in $A$ during the period $dt$ is
$dt\times\m(A)$.
Using this fact, for each $i\in\Z$, we can construct disjoint sets $\{\Delta_{ij}(\phi), j\ne i\}$
such that $\m(\Delta_{ij}(\phi))=q_{ij}(\phi).$
Let $\tilde\p$ govern the switching of $\alpha(t)$ in the manner that
if $\alpha(t-)=i$ and there is a Poisson point in $\Delta_{ij}(X_t)$ at time $t$, then $\alpha(t)=j$.
If $\alpha(t-)=i$ and there is no Poisson point in $\cup_{j\ne i}\Delta_{ij}(X_t)$ at time $t$,
 $\alpha(t)$ remains $i$.
Using this idea, we formulate the equation for $\alpha(t)$ as follows.
For each function $\phi: [-r,0]\to\R^n$, and $i\in\N$,  let $\Delta_{ij}(\phi), j\ne i$  be the consecutive left-closed
 and right-open intervals of the real line, each having length $q_{ij}(\phi)$.
That is,
\bea \ad \Delta_{i1}(\phi)=[0,q_{i1}(\phi)),\
 \Delta_{ij}(\phi)=\Big[\sum_{k=1,k\ne i}^{j-1}q_{ik}(\phi),\sum_{k=1,k\ne i}^{j}q_{ik}(\phi)\Big), j>1, j\ne i.\eea
Define $h:\C\times\N\times\R\mapsto\R$ by
$h(\phi, i, z)=\sum_{j=1, j\ne i}^\infty(j-i)\1_{\{z\in\Delta_{ij}(\phi)\}}$,
where $\1_{\{z\in\Delta_{ij}(\phi)\}}=1$ if $z\in \Delta_{ij}$, otherwise $\1_{\{z\in\Delta_{ij}(\phi)\}}=0$,  is the indicator function.
The process $\alpha(t)$ can be defined as a solution to
$$d\alpha(t)=\int_{\R}h(X_t,\alpha(t-), z)\p(dt, dz).$$ }	
The pair $(X(t),\alpha(t))$ is therefore a solution to the system of equations
\begin{equation}\label{e2.3}
\begin{cases}
dX(t)=b(X(t), \alpha(t))dt+\sigma(X(t),\alpha(t))dW(t) \\
d\alpha(t)=\disp\int_{\R}h(X_t,\alpha(t-), z)\p(dt, dz).\end{cases}
\end{equation}
A strong solution to  \eqref{e2.3} on $[0,T]$ with initial data $(\phi,i_0)$
being
$\C\times\N$-valued and $\F_0$-measurable random variable,
is an $\F_t$-adapted process $(X(t),\alpha(t))$ such that
\begin{itemize}
  \item $X(t)$ is continuous and $\alpha(t)$ is cadlag (right continuous with left limits) almost surely (a.s.).
  \item $X(t)=\phi(t)$ for $t\in[-r,0]$ and  $\alpha(0)=i_0$
  \item $(X(t),\alpha(t))$ satisfies \eqref{e2.3} for all $t\in[0,T]$ a.s.
\end{itemize}
{
We will show in the Appendix that the solution $(X(t),\alpha(t))$ to \eqref{e2.3},
 satisfies \eqref{eq:tran} under suitable conditions.}
Let  $f(\cdot,\cdot): \R^n\times\N\mapsto\R$ be twice continuously differentiable in $x$ and bounded in $(x,i)\in\R^n\times\N$.
We define the ``operator" $\LL f(\cdot, \cdot): \C\times\N\mapsto \R$ by
\begin{equation}\label{eq:oper-def}
\begin{aligned}
\LL f(\phi, i)=&\nabla f(\phi(0),i)b(\phi(0),i)+\dfrac12\trace\Big(\nabla^2 f(\phi(0),i)A(\phi(0),i)\Big)\\
&\ +\sum_{j=1,j\ne i}^\infty q_{ij}(\phi)\big[f(\phi(0),j)-f(\phi(0),i)\big]\\
=&
\sum_{k=1}^nb_k(\phi(0),i)f_k(\phi(0),i)
+\dfrac12
\sum_{k,l=1}^na_{kl}(\phi(0),i)f_{kl}(\phi(0),i)\\
&\ +\sum_{j=1,j\ne i}^\infty q_{ij}(\phi)\big[f(\phi(0),j)-f(\phi(0),i)\big],
\end{aligned}
\end{equation}
where $b(x,i)=(b_1(x,i),\dots,b_n(x,i))^\top$, $\nabla f(x,i)=(f_1(x,i),\dots,f_n(x,i))\in \rr^{1\times n}$
and $\nabla^2 f(x,i)=(f_{ij}(x,i))_{n\times n}$ are the gradient and Hessian of $f(x,i)$ with respect to $x$, respectively,
 with \bea \ad f_k(x,i) =(\partial /\partial x_k) f(x,i),\
f_{kl} (x, i) = (\partial^2/\partial x_k \partial x_l) f(x,i),\
\hbox{ and }\\
\ad A(x,i)=(a_{kl}(x,i))_{n\times n}=\sigma(x,i)\sigma^\top(x,i),\eea
with $z^\top$ denoting the transpose of $z$.
Suppose that $(X(t),\alpha(t))$ satisfies \eqref{e2.3} and that
for any $T>0$,
\begin{equation}\label{q.bound}
\sup\limits_{t\in[0,T]}\{q_{\alpha(t)}(X_t)\}<\,\infty\text{ a.s.}
\end{equation}
Let  $\xi_k=\inf\{t>0: q_{\alpha(t)}(X_t)\geq k\}, k\in\N.$
By noting that $h(X_t, z)=0$ if $z\notin [0, q_{\alpha(t)}(X_t))$
and that
$$\int_{\R}\big[f\big(\phi(0),i+h(\phi,i, z)\big)-f(\phi(0)),i)\big]\m(dz)=\sum_{j=1,j\ne i}^\infty q_{ij}(\phi)\big[f(\phi(0),j)-f(\phi(0),i\big],$$
 we have
from It\^o's formula (see \cite[Theorem 4.4.7]{DA}) that
$$
\begin{aligned}
f(X(&t\wedge\xi_k),\alpha(t\wedge\xi_k))-f(X(0),\alpha(0))\\
=&\int_0^{t\wedge\xi_k}\LL f(X_s,\alpha(s-))ds+\int_0^{t\wedge\xi_k}\nabla f(X(s),\alpha(s-))\sigma(X(s),\alpha(s-))dW(s)\\
&+\int_0^{t\wedge\xi_k}\int_0^k\big[f\big(X(s),\alpha(s-)+h(X_s,\alpha(s-), z)\big)-f(X(s),\alpha(s-))\big]\mu(ds,dz)\\
=&\int_0^{t\wedge\xi_k}\LL f(X_s,\alpha(s-))ds+\int_0^{t\wedge\xi_k}\nabla f(X(s),\alpha(s-))\sigma(X(s),\alpha(s-))dW(s)\\
&+\int_0^{t\wedge\xi_k}\int_{\R}\big[f\big(X(s),\alpha(s-)+h(X_s,\alpha(s-), z)\big)-f(X(s),\alpha(s-))\big]\mu(ds,dz),
\end{aligned}
$$
where $\mu(ds,dz)$ is the compensated Poisson random measure given by
$$\mu(ds,dz)=\p(ds,dz)- \m(dz)ds .$$
Under condition \eqref{q.bound},
there exists a random integer $k_0=k_0(\omega)$ such that
$t\wedge\xi_k=t$ for any $k>k_0$.
As a result,
\begin{equation}\label{e2.4-a}
f(X(t),\alpha(t))-f(X(0),\alpha(0))=\int_0^t\LL f(X_s,\alpha(s-))ds+M_1(t)+M_2(t)\,\text{ a.s.,}
\end{equation}
where $M_1(\cdot)$ and $M_2(\cdot)$ are local martingales, defined by
\begin{equation}\label{eq:M1-M2}
\barray \ad
M_1(t)=\int_0^t\nabla f(X(s),\alpha(s-))\sigma(X(s),\alpha(s-))dW(s),
\\
\ad
M_2(t)=\int_0^t\int_\R\big[f\big(X(s),\alpha(s-)+h(X_s,\alpha(s-), z)\big)-f(X(s),\alpha(s-))\big]\mu(ds,dz).
\earray \end{equation}
It should be noted that $\LL$ is not the generator of the Markov process
$(X_t,\alpha(t))$. However this operator is very useful for
analyzing the process $(X(t),\alpha(t))$.
In view of \eqref{eq:M1-M2}, if $\tau_1\leq\tau_2$ are
 stopping times that are bounded above by $T$ a.s.,
and $f(\cdot,\cdot)$ and $\LL f(\cdot,\cdot)$ are bounded and \eqref{q.bound} holds, then
$$\E f(X(\tau_2),\alpha(\tau_2))=\E f(X(\tau_1),\alpha(\tau_1))+\E\int_{\tau_1}^{\tau_2}\LL f(X_t,\alpha(t-))dt.$$

\begin{rem}\label{sw-d}{\rm
If  $\alpha(t)$ depends on the continuous state, but there is no past dependence (that is,  $X_t$ is replaced by $X(t)$ in \eqref{eq:tran}, and
 $\phi$ and $\phi(0)$ are replaced by the current state $X(t)=x$ in \eqref{eq:oper-def}, respectively),
then  $\LL$ is indeed  the generator of the process $(X(t),\alpha(t))$.
Even in this case, the current paper settles the matter of the state space of the switching process being countable thus
 generalizes the study of finite state space cases as considered in \cite{YZ}.
}\end{rem}

\begin{exm}\label{exm:2} {\rm
This example stems from applications in  ecological systems and biological control.
Consider the evolution of two interacting species. One is micro,
which is
described by a logistic differential equation perturbed by a white noise.
The other is macro, we assume that its number of individuals follows a birth-death process. Let $X(t)$ be the density of the micro species and $\alpha(t)$ the population of the macro species.
{ The life cycle of a micro species is usually very short,
so it is reasonable to assume that the evolution
 of $X(t)$ can be described by the following past-independent equation}
\begin{equation}\label{ex2.1}
dX(t)=X(t)\big[a(\alpha(t))-b(\alpha(t))X(t)\big]dt+\sigma(\alpha(t))X(t)dW(t),
\end{equation}
where $a(i), b(i), \sigma(i)$ are positive constants for each $i\in\Z_+$.

{ On the other hand, the reproduction process of $\alpha(t)$ is assumed
 to be non-instantaneous.
More precisely, suppose the reproduction depends on the period of time from egg formation to hatching, say $r$.
Then we have}
\begin{equation}\label{ex2.2}
d\alpha(t)=\int_{\R}h(X_t,\alpha(t-), z)\p(dt, dz),
\end{equation}
where
$h(\phi, i, z)=\sum_{j=1, j\ne i}^\infty(j-i)\1_{\{z\in\Delta_{ij}(\phi)\}}$,
$\Delta_{i,i+1}(\phi)=[0,\beta_i(\phi))$, $\Delta_{i,i-1}(\phi)=[0,\delta_i(\phi))$,
$\Delta_{i,j}(\phi)=\emptyset$ if $j\notin \{i-1,i,i+1\}$ or $i=0$.
{ Usually
$\beta_i(\phi),\delta_i(\phi)$ can be given in the integral from $\beta_i(\phi)=\int_{-r}^0\tilde\beta_i(t)\phi(t)dt$,
$\delta_i(\phi)=\int_{-r}^0\tilde\delta_i(t)\phi(t)dt$,
for some appropriate weighting functions $\tilde\beta_i, \tilde\delta_i$.}
As can be seen from the above, the switching process at $t$ in fact depends on past history of the state $X(\cdot)$.
Investigating the interactions between the two species are very important to biological control. A basic
biological control problem
aims to choose a suitable living organism to control a
particular pest (see e.g., \cite{HHK, RL}). This chosen organism might be a predator, parasite, or disease, which will attack
the harmful insect.
To design and evaluate effectiveness of a biological control, some questions should be answered first.
For example, under which conditions the species will be permanent forever or  they will extinct at some instance? Whether or not there is an invariant measure associated with the system under consideration. Mathematically, these questions are related to the stability and ergodicity of the
corresponding stochastic systems, which will be studied in a future paper.
}
\end{exm}

\begin{exm}\label{exm:3} {\rm
Pollution management is vitally important and has a significant impact on environment.
 A major issue is concerned with the tradeoff of pollution accumulation and consumption, which affects
environmental policy making.
Following
the seminal paper of Keeler et al. \cite{Keeler}, much work has been devoted to the study of optimal control of dynamic economic systems.
In \cite{morimoto}, Kawaguchi and Morimoto treated a pollution accumulation problem of maximizing the long-run
average welfare using a controlled diffusion model. 
Assume that an economy consumes some good and meanwhile generates  pollution. The pollution stock is
gradually degraded and its instantaneous growth rate incorporates a random disturbance
with mean zero and constant variance. The social welfare is defined by the
utility of the consumption net of the disutility of pollution. The problem is to find
optimal consumption strategies for the society in the long-run average sense.
Departing from their formulation, we consider an extension of their model.
Suppose that there is a switching process $\al(t)$ taking values in $\N$ such that $\al(t)$ represents the level of pollution at time $t$.
Assume that
the stock of pollution at time $t$ is given by $X(t)$, a real-valued process,
and  there is a positive real-valued function $\rho(\cdot)$ so that
for each $i\in \N$,
the  rate of pollution decay is $\rho(i)$.
The  consumption rate (or flow of pollution) is a control process, which is  denoted by $c(t)$ at time $t$;
the social utility function of the consumption $c$ is denoted by $U(c)$, whereas
 the social disutility of the pollution stock $x$ is $D(x)$. We say that
the consumption rate is admissible if it is ${\cal F}_t$-measurable, where ${\cal F}_t=\{(X(s),\al(s)): s \le t\}$ such that
$0\le c(t) \le K_0$ for some $K_0>0$.
 The ultimate objective is to maximize the long-run average welfare
 \begin{equation}\label{eq:obj} J(c\cd)= \liminf_{T\to \infty} {1\over T} \E \int^T_0 [U(c(t))- D(X(t))] dt,\end{equation}
 subject to
\begin{equation}\label{eq:sys} dX(t)= [c(t) - \rho(\al(t)) X(t)] dt + \sigma (X(t),\al(t)) dW(t).\end{equation}
Assume that the pollution level $\al\cd$ satisfies the conditions \eqref{eq:tran}. First, it is reasonable
that the level of pollution can be modeled by
a continuous-time process taking values in $\N$. Second, to be more realistic, the pollution level depends on the
pollution stock $X(t)$ as well as some past history as given in \eqref{eq:tran}.
As another generalization of \cite{morimoto}, we assume that $\sigma$ in fact depends on
$(X(t),\al(t))$,
and the switching rate depends on some past history of the pollution
stock $X\cd$ as in \eqref{eq:tran},
and $\sigma^2 (x,i)>0$ for each $i\in \N$.
Treating the optimal pollution management problem, it is natural to consider the replacement of the average in \eqref{eq:obj} by the
average with respect to an invariant measure (if it exists) of the controlled systems. To do so, we need to make sure that \eqref{eq:sys} indeed has an invariant measure.
Before this matter can be settled, we need to show that the system has a unique solution for each initial data,
and the solution possesses certain desired properties such as Markov and Feller properties.
}\end{exm}

\section{Existence and Uniqueness of Solutions}\label{sec:exi}
We are now in a position to prove the existence and uniqueness in the strong sense of a  solution with  given initial
data under suitable conditions.
{ We give several sets of conditions. The main reason is due to the past dependence and the use of $\N$. First in contrast to the case of switching process staying in a finite set, care needs to be exercised regarding uniformity with respect to the switching set. Second, the past dependence requires careful handling of the use of Lipschitz continuity etc. and the uniformity with respect to the element in the corresponding function spaces. Depending on the preference,
Assumptions 3.1 allows certain bounds to be dependent of the switching state $i$, but uniform in the variable in the function
 space, whereas Assumption 3.2 requires uniformity in the bounds w.r.t. $i$, but requires the past dependent part be localized. Assumptions 3.3 and 3.4 relax the Lipschitz condition to local Lipschitz together with certain growth conditions presented by using   bounds with the help of Lyapunov functions. }

\begin{asp}\label{asp2.1} {\rm Assume the following conditions hold.
\begin{itemize}
  \item[{(i)}] For each $i\in\N$, there is a positive constant $L_i$ such that
$$|b(x,i)-b(y,i)|+|\sigma(x,i)-\sigma(y,i)|\leq L_i|x-y|\,\forall x,y\in\R^n.$$
\item[{(ii)}] $q_{ij}(\phi)$ is measurable in $\phi\in\C$ for all $i$ and $j\in\N$. Moreover,
      $$M:=\sup_{\phi\in\C, i\in\N}\{q_{i}(\phi)\}<\infty.$$
\end{itemize} }
\end{asp}

\begin{thm}\label{thm2.1}
Under Assumption {\rm\ref{asp2.1}}, for each
 initial data $(\xi,i_0)$, there exists a unique solution $(X(t),\alpha(t))$ to \eqref{e2.3}.
\end{thm}

\begin{proof}
It is well-known that part (i) of Assumption \ref{asp2.1} guarantees the existence and uniqueness of strong solutions to
the following diffusion
\begin{equation}\label{e2.4t}
dY(t)=b(Y(t), i)dt+\sigma(Y(t),i)dW(t) \ \text{ for each }\  i\in\Z_+.
\end{equation}
Then,
given a stopping time $\tau$ and an $\F_\tau$-measurable $\R^n$-valued random variable $y=y(\tau)$ (depending on $\tau$),
there exists a unique strong solution to \eqref{e2.4t} in $[\tau,\infty)$ satisfying $Y(\tau)= y(\tau)$
(see \cite[Remark 3.10]{MY}).
We can now  construct the solution to \eqref{e2.3} with initial data $(\xi, i_0)$ by the interlacing procedure similar to \cite[Chapter 5]{DA}.
Let $\tilde Y^{(0)}(t), t\geq0$ be the solution with initial data $\xi(0)$ to
\bed\label{e2.5}
d\tilde Y^{(0)}(t)=b(\tilde Y^{(0)}(t), i_0)dt+\sigma(\tilde Y^{(0)}(t),i_0)dW(t).
\eed
We also set $\tilde Y^{(0)}(t)=\xi(t)$ for $t\in[-\tau,0]$.
Let \bea \ad \tau_1=\inf\{t>0: \int_0^t\int_\R h(\tilde Y^{(0)}_s, i_0, z)\p(ds,dz)\ne 0\}
\ \hbox{ and } \\
\ad i_1=i_0+\int_0^{\tau_1}\int_\R h(\tilde Y^{(0)}_s, i_0, z)\p(ds,dz),\eea
and
$\tilde Y^{(1)}(t), t\geq\tau_1$ be the solution with $\tilde Y^{(1)}_{\tau_1}=\tilde Y^{(0)}_{\tau_1}$
to
\begin{equation}\label{e2.6}
d\tilde Y^{(1)}(t)=b(\tilde Y^{(1)}(t), i_1)dt+\sigma(\tilde Y^{(1)}(t),i_1)dW(t).
\end{equation}
Define
\bea \ad \tau_2=\inf\{t>\tau_1: \int_{\tau_1}^t\int_\R h(\tilde Y^{(1)}_s, i_1, z)\p(ds,dz)\ne 0\} \
\hbox{ and } \\
\ad i_2=i_1+\int_{\tau_1}^{\tau_2}\int_\R h(\tilde Y^{(1)}_s, i_1, z)\p(ds,dz).\eea
Note that, in the notation above, $\tilde Y^{(k)}_t$ is the function $s\in[-r,0]\mapsto \tilde Y^{(k)}(t+s).$
Continuing this procedure, let $\tau_\infty=\lim\limits_{k\to\infty}\tau_k$ and set
\begin{equation}\label{e2.7}
X(t)=\tilde Y^{(k)}(t), \ \alpha(t)=i_k\text{ if } \tau_k\leq t<\tau_{k+1}.
\end{equation}
Clearly, $X(t)$ satisfies that for every $t\geq0$,
\begin{equation}\label{e2.8}
\begin{cases}
X(t\wedge\tau_k)=X(0)+\disp\int_0^{t\wedge\tau_k}\big[b(X(s), \alpha(s))ds+\sigma(X(s),\alpha(s))dW(t)\big] \\
\alpha(t\wedge\tau_k)=i_0+\disp\int_0^{t\wedge\tau_k}\int_{\R}h(X_s,\alpha(s-), z)\p(ds, dz).\end{cases}
\end{equation}
To show that $X(t)$ is a global solution, we  need only
prove that $\tau_\infty=\infty$ a.s.
For any $T>0$,
\begin{equation}\label{e2.9}
\begin{aligned}
\PP\{\tau_k\leq T\}=&\PP\big\{\int_0^{T\wedge\tau_k}\int_{\R}\1_{\{z\in[0,q_{\alpha(s-)}(X_s))\}}\p(ds, dz)= k\big\}\\
\leq&\PP\big\{\int_0^{T\wedge\tau_k}\int_{\R}\1_{\{z\in[0,M)\}}\p(ds, dz)\geq k\big\}\\
\leq&\PP\big\{\int_0^{T}\int_{\R}\1_{\{z\in[0,M)\}}\p(ds, dz)\geq k\big\}\\
=&\sum_{l=k}^\infty e^{-MT}\dfrac{(MT)^l}{l!}.
\end{aligned}
\end{equation}
It follows that $\PP\{\tau_k\leq T\}\to0$ as $k\to\infty$.
As a result $\tau_\infty=\infty$ a.s.
By this construction, it can be seen that $X(t)$ is continuous and $\alpha(t)$ is cadlag almost surely.
The uniqueness of $(X(t),\alpha(t))$ follows from the uniqueness of $\tilde Y^{(k)}(t)$ on $[\tau_{k},\tau_{k+1}]$ and the uniqueness of
$i_k$ defined by $$i_k=i_{k-1}+\int_{\tau_{k-1}}^{\tau_k}\int_\R h(\tilde Y^{(k-1)}_t, i_{k-1}, z)\p(dt,dz).$$
This concludes the proof.
\end{proof}

\begin{asp}\label{asp2.2} {\rm Assume the following conditions hold.
\begin{itemize}
  \item[{(i)}]  There is a positive constant $L$ such that
$$|b(x,i)-b(y,i)|+|\sigma(x,i)-\sigma(y,i)|\leq L|x-y|,\  \forall x,y\in\R^n, i\in\N.$$
\item[{(ii)}]  $q_{ij}(\phi)$ is measurable in $\phi\in\C$ for each $(i,j)\in\N^2$. Moreover, for any $H>0$,
      $$M_H:=\sup_{\phi\in\C, \|\phi\|\leq H, i\in\N}\{q_{i}(\phi)\}<\infty.$$
\end{itemize} }
\end{asp}

{
\begin{rem}\rm
We can use either Assumption \ref{asp2.1} or Assumption \ref{asp2.2}
to obtain the existence and uniqueness of solutions to \eqref{e2.3}.
Recall that now $\N$ is a countable set, so care must be taken to distinct it with a finite state case.
In Assumption \ref{asp2.1}, the Lipschitz constants of $b(\cdot, i),\sigma(\cdot,i)$ depend on
$i$, and $q_i(\phi)$ is assumed to be bounded uniformly in $(\phi,i)\in\C\times\Z_+$.
In contrast, the uniform boundedness of $q_i(\phi)$ is
relaxed,
but the Lipschitz constant of $b(\cdot, i),\sigma(\cdot,i)$ is assumed to be
in $i\in\N$.
\end{rem}
}

\begin{thm}\label{thm2.2}
Under Assumption {\rm\ref{asp2.2}}, for each
initial data $(\xi,i_0)$, there exists a unique solution $(X(t),\alpha(t))$ to \eqref{e2.3}.
\end{thm}

\begin{proof}
Without loss of generality, we may assume that $(\xi, i_0)$
is bounded,
since we can use the truncation method in \cite[Theorem 3 in \S6]{GS} to obtain the result for general $(\xi, i_0)$ once we have proved for the case $(\xi, i_0)$ being bounded.
Construct the process $(X(t),\alpha(t))$ as in the proof of Theorem \ref{thm2.1}.
We need to show that $\tau_\infty=\infty$ a.s.
Following the proof of \cite[Lemma 3.2, p. 51]{XM}, there is a $K=K(T)$ such that
$$\E\Big(\sup\limits_{0\leq t\leq T\wedge\tau_k}|X(t)|^2\Big)\leq K\,\forall\,k\in\N.$$
As a result, for any $\eps>0$, there is an $H_\eps$ such that
\begin{equation}\label{e2.10}
\PP\{\|X_{t}\|\leq H_\eps\,\forall\,t\in[0,T\wedge\tau_k]\}>1-\dfrac\eps2.
\end{equation}
Let $\eta_{H_\eps}=\inf\{t\geq0: \|X_{t}\|\geq H_\eps\}$
and $M_{H_\eps}=\sup_{\phi\in\C, \|\phi\|\leq H_\eps, i\in\N}\{q_{i}(\phi)\}<\infty.$ Then
\begin{equation}
\begin{aligned}
\PP\{\tau_k\leq T\wedge\eta_{H_\eps}\}=&\PP\big\{\int_0^{T\wedge\tau_k\wedge\eta_{H_\eps}}\int_{\R}\1_{\{z\in[0,q_{\alpha(s-)}(X_s))\}}\p(ds, dz)= k\big\}\\
\leq&\PP\big\{\int_0^{T\wedge\tau_k\wedge\eta_{H_\eps}}\int_{\R}\1_{\{z\in[0,M_{H_\eps})\}}\p(ds, dz)\geq k\big\}\\
\leq&\PP\big\{\int_0^{T}\int_{\R}\1_{\{z\in[0,M_{H_\eps})\}}\p(ds, dz)\geq k\big\}\\
=& e^{-M_{H_\eps} T}\sum_{l=k}^\infty \dfrac{(M_{H_\eps} T)^l}{l!}.
\end{aligned}
\end{equation}
For sufficiently large $k$, we have
\begin{equation}\label{e2.12}
\PP\{\tau_k\leq T\wedge\eta_{H_\eps}\}\leq e^{-M_{H_\eps} T}\sum_{l=k}^\infty \dfrac{(M_{H_\eps} T)^l}{l!}\leq\dfrac\eps2.
\end{equation}
From \eqref{e2.10} and \eqref{e2.12},
$\PP\{\tau_k\geq T\}\geq\PP(\{\tau_k\wedge T<\eta_{H_\eps}\}\cap\{\tau_k> T\wedge\eta_{H_\eps}\})\geq1-\eps$
for sufficiently large $k$.
Thus, we
obtain that $\PP\{\tau_\infty\geq T\}\geq 1-\eps$.
It holds for every $T>0$ and $\eps>0$, so we obtain the desired result.
\end{proof}

\begin{rem}{\rm
To obtain the existence and uniqueness of solutions, Assumptions \ref{asp2.1} and \ref{asp2.2} can be relaxed by replacing the
global Lipschitz conditions with
local Lipschitz conditions together with
Lyapunov-type functions. To be specific,
let  $V(\cdot): \R^n\mapsto\R$ be  twice continuously differentiable in $x$.
For each $i\in\N$, let $\LL_i V(x)=\nabla V(x)b(x,i)+\dfrac12\trace\Big(\nabla^2 V(x)A(x,i)\Big)$.
For instance (1) of Assumption \ref{asp2.1} and (1) of Assumption \ref{asp2.2} can be replaced by the following
Assumptions \ref{asp2.3} and \ref{asp2.4}, respectively.
}\end{rem}

\begin{asp}\label{asp2.3}{\rm Assume the following conditions hold.
\begin{itemize}
  \item[{(i)}] For each $H>0$, $i\in\N$, there is a positive constant $L_{H,i}$ such that
$$|b(x,i)-b(y,i)|+|\sigma(x,i)-\sigma(y,i)|\leq L_{H,i}|x-y|, \ \forall |x|, |y|\leq H, i\in\N.$$
\item[{(ii)}] For each $i\in\N$, there exist a  twice continuously differentiable function $V_i(x)$ and a constant $C_i>0$ such that
$$\lim\limits_{R\to\infty}\Big(\inf\{V_i(x):|x|\geq R\}\Big)=\infty\quad\text{ and }\quad\LL_iV_i(x)\leq C_i(1+V_i(x))\,\forall\,x\in\R^n.$$
\end{itemize} }
\end{asp}

\begin{asp}\label{asp2.4} {\rm Assume the following conditions hold.
\begin{itemize}
  \item[{(i)}] For each $H>0$, $i\in\N$, there is a positive constant $L_{H,i}$ such that
$$|b(x,i)-b(y,i)|+|\sigma(x,i)-\sigma(y,i)|\leq L_{H,i}|x-y|\, \ \forall |x|, |y|\leq H, i\in\N.$$
\item[{(ii)}] There exist a  twice continuously differentiable function $V(x)$ and a constant $C>0$ independent of $i\in\N$ such that
$$\lim\limits_{R\to\infty}\Big(\inf\{V(x):|x|\geq R\}\Big)=\infty\quad\text{ and }\quad\LL_iV(x)\leq C(1+V(x))\,\forall\,x\in\R^n, i\in\N.$$
\end{itemize}}
\end{asp}

\begin{thm}\label{thm2.3}
For  given initial data $(\xi,i_0)$, there exists a unique solution $(X(t),\alpha(t))$ to \eqref{e2.3} if either of the following conditions  is satisfied
\begin{itemize}
\item Assumption {\rm\ref{asp2.3}} and {\rm (ii)} of Assumption {\rm\ref{asp2.1}},
\item Assumption {\rm\ref{asp2.4}} and {\rm(ii)} of Assumption {\rm\ref{asp2.2}}.
\end{itemize}
\end{thm}

\begin{proof}
It is well known that  Assumption \ref{asp2.3} guarantees the existence and uniqueness of solutions to
\eqref{e2.4t}. Hence, if
(ii) in Assumption \ref{asp2.1} is satisfied, we can prove
the desired result by
using the proof of Theorem \ref{thm2.1}.
Now, suppose  Assumption \ref{asp2.4} and
(ii) of Assumption \ref{asp2.2} hold.
Similar to
 the proof of Theorem \ref{thm2.2}, we can assume that $(\xi, i_0)$ is bounded.
Consider $X(t)$ and define $\tau_k$ as in the proof of Theorem \ref{thm2.1}.
Then $X(t)$ is the solution with initial data $(\xi,i_0)$ to \eqref{e2.3} on $[0,T\wedge\tau_k)$ for any $T>0, k\in\N$.
We have from the generalized It\^o formula that
$$
\begin{aligned}
\E V(X(T\wedge\tau_k\wedge\eta_H))& = \E V(\xi(0),i_0)+\E\int_0^{T\wedge\tau_k\wedge\eta_H}\LL_{i} V(X(t), \alpha(t-))dt\\
&\leq \E V(\xi(0),i_0)+C\E\int_0^{T\wedge\tau_k\wedge\eta_H}(1+V(X(t))dt,
\end{aligned}
$$
where $\eta_H=\inf\{t\geq0: |X(t)|>H\}$.
Using the estimate above and the argument in \cite[Theorem 3.19]{MY}, we can show that
$$\E V(T\wedge\tau_k\wedge\eta_H)\leq K=K(\xi,T)\,\forall H>0, k\in\N.$$
In view of the property $\lim\limits_{R\to\infty}\Big(\inf\{V(x):|x|\geq R\}\Big)=\infty$, for any $\eps>0$, there is $H_\eps>0$ such that
$$\PP\{\eta_{H_\eps}>T\wedge\tau_k\}>1-\dfrac\eps2\,\,\forall\, k\in\N.$$
Then,
proceeding similarly as in the proof of Theorem \ref{thm2.2} yields the existence and uniqueness of solutions with initial data $(\xi,i_0)$ to \eqref{e2.3}.
\end{proof}

\begin{exm}
\label{exm:2a}{\rm (cont. of Example \ref{exm:2})
We come back to Example \ref{exm:2} in Section \ref{sec:int}.
We want to show that
$X(t)>0$ for all $t\geq0$ under certain conditions. To proceed, we can set $Y(t)=\ln X(t)$ to obtain
\begin{equation}\label{e.lnx}
d Y(t)=[a(\alpha(t)-\dfrac{\sigma^2(\alpha(t))}2-b(\alpha(t))\exp(Y(t))]dt+\sigma(\alpha(t))dW(t).
\end{equation}
To demonstrate  \eqref{ex2.1} and \eqref{ex2.2} has a unique solution with $X(t)>0$ for all $t\geq0$, it is equivalent to show that \eqref{e.lnx} and \eqref{ex2.2}
has a strong solution on $[0,\infty)$.
Let $V(y)=e^y+e^{-y}$.
By direct calculation,
$$
\begin{aligned}
\LL_i V(y)=& b(i)+(\sigma^2(i)-a(i))e^{-y}+a(i) e^{y}-b(i)e^{2y}\\
\leq& c(i)+(\sigma^2(i)-a(i))V(y),
\end{aligned}
$$
where $c(i)=\max\limits_{y\in\R}\{b(i)+(2a(i)-\sigma^2(i))e^{y}-b(i)e^{2y}\}$.
Applying Theorem 3.3, we can see that the equation has a unique solution if one of the following is satisfied
\begin{itemize}
  \item $\beta_i(\phi)+\delta_i(\phi)$ is bounded uniformly in $\phi\in \C_+:=\{\psi\in \C: \psi(t)>0\,\forall t\in[-r,0]\}$ and $i\in\Z_+$.
  \item $c(i)$ and $\sigma^2(i)-a(i)$ are bounded above uniformly and for each $i\in\Z_+$, $\beta_i(\phi)+\delta_i(\phi)$ is bounded in each compact subset of $\phi\in \C_+$.
\end{itemize}
It can be shown by applying the result of the next section that  the
process $(Y_t, \alpha(t))$ has the Markov-Feller property if $\beta_i(\cdot)$ and $\delta_i(\cdot)$ are continuous in addition to one of the above conditions.
}\end{exm}

\begin{exm}\label{exm:3a} {\rm (cont. of Example \ref{exm:3})
To study the long-run average control problem in Example \ref{exm:3}, it is important to make sure that the system under consideration processes ergodicity.
Before the ergodicity can be verified, we need \eqref{eq:sys} has a unique solution for each initial condition.
Denote the control set by $\wdt K$ and assume it is a compact and convex set.
Using a relaxed control representation $m_t\cd$ (see \cite{Kushner90}) to represent the consumption rate
$c\cd$, we can rewrite \eqref{eq:sys} as
\begin{equation}\label{eq:sys-1}
 dX(t)= \Big[ \int_{\wdt K} c(u) m_t(du)  - \rho(\al(t) ) X(t) \Big] dt +\sigma(X(t),\al(t)) dW(t).\end{equation}
Assume that for each $i\in \N$, $\sigma(x,i)$ satisfies the conditions in
Assumption \ref{asp2.1} (i), and $Q(\phi)$ satisfies Assumption \ref{asp2.1} (ii).
Then the conditions of Theorem \ref{thm2.1} are all verified.
As a result, \eqref{eq:sys-1} has a unique solution for each initial condition.
}\end{exm}

\section{Markov and Feller Properties}\label{sec:mar}
{ This section establishes the Markov and Feller properties of the process $(X_t,\alpha(t))$.
While the Markov property can be derived by the well-known arguments, it requires much more efforts to obtain
the Feller property.
As already seen in the previous section, the past dependence and the use of $\N$ make the analysis more complex than that of the switching diffusions with diffusion-dependent switching living in a finite set. To overcome the difficulties, in this section, we carry out the analysis by introducing some auxiliary or intermediate processes.
First, it would be better if we could untangle the past dependence of the switching process and the infinity of the cardinality of its state space. For this purpose, we introduce a continuous-time Markov chain independent of the past and continuous state; we call this process $\gamma(t)$. Then naturally, associated with $\gamma(t)$,
we examine a pair of process $(Z(t),\gamma(t))$. Even after this introduction, in the analysis, we still need to look into the details of the switching process $\alpha(t)$ such as when it jumps and the post jump location etc. To do so, we introduce another auxiliary process $Y(t)$, which is a ``fixed''-$i$ process.
We then have another pair of processes $(Y(t),\beta(t))$ to deal with.  These auxiliary processes help us to establish the desired results.
Their connections and interactions will be further specified in what follows.}

First, note that the Brownian motion and the Poisson point process associated to $\p(dt,dz)$ possess stationary strong Markov property, that is,
for any finite stopping time $\eta$, $\{W^*(t)\}_{t\geq0}=\{W(t+\eta)-W(\eta)\}_{t\geq0}$ is an $\F^*_t$-Brownian motion and
 $\p^*([t,t+s)\times U)=\p([t+\eta,t+s+\eta)\times U)$ is a Poisson random measure with density $dt\times\m(dz)$ (see \cite[Theorem 101]{SR}).
Hence, by standard arguments, we can obtain the following theorem whose proof is omitted.
In fact, the theorem can be proved essentially by imitating the proof in \cite[Chap. 5]{PP}, \cite[Chap. 6]{DA}, or \cite[Chap. 7]{BO}.

\begin{thm}\label{thm3.1}
Assume that
the hypotheses of Theorem {\rm\ref{thm2.1}},  or  Theorem {\rm\ref{thm2.2}}, or Theorem {\rm\ref{thm2.3}}
are satisfied.
Let $(X(t),\alpha(t))$ be a solution to \eqref{e2.3}.
Then $(X_t,\alpha(t))$ is a homogeneous strong Markov process taking value in $\C\times\N$ with transition probabilities
$$P(\phi,i, t, A\times\{j\})=\PP\{X^{\phi,i}_t\in A, \alpha(t)=j\},$$
where $X^{\phi,i}(t)$ is the solution to \eqref{e2.3} with initial data $(\phi,i)\in\C\times\N$.
\end{thm}

We proceed with  obtaining the Feller property of $(X_t,\alpha(t))$. Assuming that
the hypotheses of Theorem \ref{thm2.1}, or Theorem \ref{thm2.2}, or Theorem \ref{thm2.3}
are satisfied leads to the existence and uniqueness of strong solutions.
Next, we introduce an auxiliary hybrid diffusion with Markov switching.
Let $\gamma^i(t)$ be a Markov chain starting at $i$ with generator $\tilde Q=(\rho_{ij})$ for $(i,j) \in {\N\times\N}$,
where $\rho_{ii}=-1$  and $\rho_{ij}=2^{-j}$ if $j<i$ and $\rho_{ij}=2^{-j+1}$ if $j>i$, that is,
$$
\tilde Q=\left(\begin{array}{cccc}
-1&1/2&1/4&\cdots\\
1/2&-1&1/4&\cdots\\
1/2&1/4&-1&\cdots\\
\vdots&\vdots&\vdots&\ddots\\
\end{array}\right).
$$
{
We recursively define a sequence of stopping times $\{\theta^i_k\}$ with $\theta^i_k$ being the first jump time of $\gamma^i(t)$ after $\theta^i_{k-1}$ as follows}
$$\theta_0^i=0, \ \theta_{k}^i=\inf\{t>\theta_{k-1}^i: \gamma^i(t)\ne\gamma^i(\theta_{k-1}^i)\}, k\in\N.$$
For $(\phi,i)\in\C\times\N$, let
$Z^{\phi,i}(t)$ be the solution to
$$dZ(t)=b(Z(t),\gamma(t))dt+\sigma(Z(t),\gamma(t))dW(t)\,,t\geq0$$
satisfying $Z^{\phi,i}(t)=\phi(t)$ in $[-r,0]$ and $\gamma(0)=i$.
{
Similar to Girsanov's theorem, which tells us how to convert an It\^o process
to a Brownian motion under a change of measure,
we aim to establish a change of measure allowing us
to ``convert" a hybrid diffusion with past-dependent switching to a hybrid diffusion with Markov switching.
To establish such a change of measure,
we need to find the distribution of jump times of $\alpha(t)$.
Because of the
interactions between $\alpha(t)$ and $X(t)$,
we need to introduce another auxiliary (or intermediate) process, which helps to examine the distribution
of the jump times of $\alpha(t)$.}
Let $(Y^{\phi,i}(t),\beta^{\phi,i}(t))$
be the solution to
\begin{equation}\label{e.3}
\begin{cases}
dY(t)=b(Y(t), i)dt+\sigma(Y(t),i)dW(t),\ t\geq0 \\
d\beta(t)=\disp\int_{\R}h(Y_t,\beta(t-), z)\p(dt, dz),t\geq0\end{cases}
\end{equation}
satisfying
 $Y^{\phi,i}(t)=\phi(t)$ in $[-r,0]$ and $\beta^{\phi,i}(0)=i$.
{
By the definition,
$\alpha^{\phi,i}(t)=\beta^{\phi,i}(t), X^{\phi,i}(t)=Y^{\phi,i}(t)$ up to the first jump time of the two process
$\alpha(t)$ and $\beta(t)$.
There is an advantage working with $(Y^{\phi,i}(t),\beta^{\phi,i}(t))$.
Unlike the pair $(X(t),\alpha(t))$ in which $\alpha(t)$ depends on the continuous state,
the process $Y^{\phi,i}(t)$ evolving for a fixed discrete state $i$ that does not depend on $\beta^{\phi,i}(t)$.
Thus, it is easier to examine, for example, the first jump time of $\beta^{\phi,i}(t)$ (or $\alpha^{\phi,i}(t)$).
}

{ Next we recursively define sequences of stopping times associated with $\beta(t)$ and $\al(t)$ so that
$\lambda^{\phi, i}_k$ and $\tau^{\phi,i}_k$ are the first jump times
of the processes $\beta^{\phi, i}(t)$ and $\al^{\phi, i}(t)$
after $\lambda^{\phi, i}_{k-1}$ and $\tau^{\phi,i}_{k-1}$, respectively. More specifically, for $k\in \N$,}
let $$\lambda_0^{\phi,i}=0, \lambda_k^{\phi,i}=\inf\{t>\lambda_{k-1}^{\phi,i}: \beta^{\phi,i}(t)\ne \beta^{\phi,i}(\lambda_{k-1}^{\phi,i})\},\ i\in\N.$$
and
$$\tau_0^{\phi,i}=0,\ \tau_k^{\phi,i}=\inf\{t>\tau_{k-1}^{\phi,i}: \alpha^{\phi,i}(t)\ne \alpha^{\phi,i}(\tau_{k-1}^{\phi,i})\},\ i\in\N.$$
To simplify the notation, we put
$$\alpha^{\phi,i}_{k}:=\alpha^{\phi,i}(\tau^{\phi,i}_k), \ \beta^{\phi,i}_{k}:=\beta^{\phi,i}(\lambda^{\phi,i}_k), \
\gamma^{i}_{k}:=\gamma^{i}(\theta^{i}_k),$$
and
$$X^{\phi,i}_{(k)}:=X^{\phi,i}_{\tau^{\phi,i}_k}, \ Y^{\phi,i}_{(k)}:=Y^{\phi,i}_{\lambda^{\phi,i}_k}, \ Z^{\phi,i}_{(k)}:=Z^{\phi,i}_{\theta^{i}_k},$$
where we use the subscript $k$ with parentheses to avoid confusion with
the function-valued processes $X^{\phi,i}_t, Y^{\phi,i}_t,  Z^{\phi,i}_t$ at $t=k$.

\begin{lm}\label{lm3.0}
Let $g:\C\times\R_+\times\N\mapsto\R$ be a bounded and measurable function, and $\F^W_T$ be the $\sigma$-algebra generated by $\{W(t), t\in[0,T]\}$.
The following assertions hold:
\begin{enumerate}[{\rm(i)}]
\item $\PP\big(\{\lambda_1^{\phi,i}> t\}\big|\F^W_T\big)=\E\Big[\1_{\{\lambda_1^{\phi,i}> t\}}\Big|\F^W_T\Big]=\exp\big(-\disp\int_0^tq_i(Y^{\phi,i}_s)ds\big)\, \ \forall\,t\in[0,T].$
\item $\E\Big[g(Y^{\phi,i}_{(1)},\lambda_1^{\phi,i},\beta^{\phi,i}_1)\1_{\{\lambda_1^{\phi,i}\leq T\}}\Big|\F^W_T\Big]=\disp\sum_{j=1, j\ne i}^\infty\int_0^Tg(Y_t,t,j)q_{ij}(Y_t)\exp(-\int_0^tq_{i}(Y_s)ds)dt.$
\end{enumerate}
\end{lm}

As indicated previously,
it is difficult to estimate the difference of $X^{\phi_1,i}_t$ and $X^{\phi_2,i}_t$ because the
states of $\alpha^{\phi_1,i}(t)$ and $\alpha^{\phi_2,i}(t)$ 
\ may differ significantly due to the continuous state dependence.
In contrast, it is considerably easier to compare $Z^{\phi_1,i}_t$ and $Z^{\phi_2,i}_t$ because of the
continuous-state-dependent switching is replaced by the continuous-state-independent Markov chain.
With help of the intermediate process $(Y(t),\beta(t))$ and
Lemma \ref{lm3.0}, we obtain
the following change of measure formula, which is a bridge to connect the continuous-state-dependent and continuous-state-independent processes.

\begin{prop}\label{prop3.2}
For any $T>0$, let $f(\cdot,\cdot):\C\times\N\mapsto\R$ be a bounded continuous function.
For any $l=0,1,\dots$,  any $i_k\in\N$ with $i_k\ne i_{k+1}$ and $k=1,\dots,l+1$,
and any $(\phi,i)\in\C\times\N$,
\begin{equation}\label{e3.1}
\begin{aligned}
\E\Big[ f&(X^{\phi,i}_T,\alpha^{\phi,i}(T))\1_{\{\tau^{\phi,i}_l\leq T<\tau^{\phi,i}_{l+1}\}}\prod_{k=1}^l\1_{\{\alpha^{\phi,i}_{k}=i_k\}}\Big]\\
=&\exp(T)\E\Big[ f(Z^{\phi,i}_T,i_l)\1_{\{\theta^{i}_l\leq T<\theta^{i}_{l+1}\}}\prod_{k=1}^l\Big(\1_{\{\gamma^{i}_{k}=i_k\}}\dfrac{q_{i_ki_{k+1}}
(Z^{\phi,i}_{(k)})}{\rho_{i_ki_{k+1}}}\Big)
\exp\Big\{-\int_0^Tq_{\gamma^{i}(s)}(Z^{\phi,i}_s)ds\Big\}\Big].
\end{aligned}
\end{equation}
\end{prop}

\begin{rem}\label{rem:no-u}{\rm
The proofs of Lemma \ref{lm3.0} and Proposition \ref{prop3.2} will be given in the Appendix. We are now in a position to prove the Feller property for the solution to \eqref{e2.3}.
 In addition to the sufficient
 conditions for the existence and uniqueness of solution, we  prove the Feller property of the solution only with an additional condition that $q_{ij}(\phi)$ is continuous in $\phi$ for any $i,j\in\N$. There are some difficulties
because the process $\{X_t\}$ takes value in an infinite dimensional Banach space and the switching $\{\alpha(t)\}$ has an infinite state space.
Moreover, although we suppose that
$q_{ij}(\phi)$ is  continuous,
neither the uniform continuity in $\phi\in\C$ nor equi-continuity in $i,j\in\N$ is assumed.
Because of these difficulties, we divide the proof into several steps. First, we make the following assumptions, which will be relaxed later.}\end{rem}

\begin{asp}\label{asp3.1}{\rm Assume the following conditions hold.
\begin{enumerate}[{\rm(i)}]
\item For each $i\in\N$, $b(x, i)$ and  $\sigma(x,i)$ are Lipschitz continuous functions that are vanishing outside $\{x: |x|\leq R\}$
for some $R>0$.
\item $M:=\sup\{q_{i}(\phi): i\in\N, \phi\in\C\}<\infty$.
\item For each $i,j\in\N, j\ne i$, $q_i(\cdot)$ and $q_{ij}(\cdot)$ are continuous on $\C$.
\end{enumerate}}
\end{asp}

Before applying \eqref{e3.1} to prove the continuous dependence of
$u_f(\phi,i)=\E_{\phi,i} f(X_T,\alpha(T))$ on $(\phi,i)$,
we first need the following lemma.

\begin{lm}\label{lm3.1}
Assume that Assumption {\rm\ref{asp3.1}} is satisfied. Let $(\phi_0,i_0)\in\C\times\N$ with $\|\phi_0\|\leq R$ and $T>0$. For each $\Delta>0$, there exist $m=m(\Delta)\in\N$,  $n_m=n_m(\Delta)\in\N$, and $d_m=d_m(\Delta)>0$ such that
$$\PP\Big(\{\tau_{m+1}^{\phi,i_0}> T\}\cap\{\alpha^{\phi,i_0}(t)\in N_{n_m},\forall t\in[0,T] \}\Big)\geq 1-\Delta,\ \ \forall \|\phi-\phi_0\|<d_m,$$
where $N_k=\{1,\dots,k\}$.
\end{lm}

{
This lemma allows us to
confine our attention to
a finite subset of $\N$ (the state space of $\alpha^{\phi,i_0}(\cdot)$) and a finite number of jumps when $\phi$ is close to $\phi_0$.
It is a crucial step in providing some uniform estimates because we
do not
assume the equi-continuity of
$q_{ij}(\cdot)$ in either $i$ or $j$.
Since the switching intensity of $\alpha^{\phi,i_0}(t)$ depends on $X_t^{\phi,i_0}$,
in order to obtain Lemma \ref{lm3.1},
we need to show that with an arbitrarily large probability,
$X_t^{\phi,i_0}, t\in[0,T]$ belongs to a compact set in $\C$ for any $\phi$  sufficiently close to $\phi_0$.
Note that under some suitable conditions,
sample paths of a diffusion process in a finite interval $[0,T]$ are H\"older continuous.
Thus, it is easy to find a compact set in which sample paths of a diffusion process lie with a large probability.
Our arguments rely on this fact.
However,  the initial data $\phi$ of our process $X(t)$ does not always satisfy the H\"older continuity.
Moreover, $X(t)$ depends on the state of $\alpha(t)$.
We therefore need to introduce the following operator, which is motivated by merging trajectories of $X(t)$ at jump times.
For $\A, \B\subset\C$, we define the set of continuous functions that are formed by merging functions in $\A$ and $\B$
as follows.
\begin{align*}
\A\uplus\B:=\A\cup\B\cup\{\psi\in \C: &\exists \psi_1\in A,\psi_2\in B, s\in[0, r]\text{ such that }\\
& \psi(t)=\psi_1(s+t)\,\forall\,t\in[-r,-s], \psi(t)=\psi_2(t+s-r)\,\forall\,t\in[-s,0]\}.
\end{align*}
By virtue of the Arzel\'a-Ascoli theorem, if $\A$ and $\B$ are compact, so is $\A\uplus\B$.		
Using this fact and the H\"older continuity of sample paths of a diffusion process,
we can find a suitable compact set to which $X_t^{\phi,i_0}, t\in[0,T],$ belongs with a large probability
for any $\phi$ which is sufficiently close to $\phi_0$.
Then, Lemma \ref{lm3.1} can be proved.
The details of the proof are postponed to the Appendix.
Now, we point out some nice properties of the diffusion process with Markov switching $(Z(t),\gamma(t))$,
which are useful to compare the sample paths of $Z(t)$ with different initial values.
}

\begin{lm}\label{lm3.2}
Fix $i_0\in\N$. For each $k\in\N$ and $\eps>0$, there is an $\hbar_k=\hbar_k(\eps)>0$
such that
$$\PP\Big\{\sup\limits_{t,s\in[0,T\wedge \iota_k],0<t-s<\hbar_k}\dfrac{|Z^{\phi, i_0}(t)-Z^{\phi,i_0}(s)|}{(s-t)^{0.25}}\leq 4\Big\}>1-\eps\,\forall\,|\phi(0)|\leq R,$$
and
$$\E\Big[\sup\limits_{t\in[0,T\wedge \iota_k]}|Z^{\phi,i_0}-Z^{\psi,i_0}|^2\Big]\leq \bar C|\phi-\psi|^2,$$
where $\iota_k=\inf\{t>0: \gamma^{i_0}(t)>k\}$ and $\bar C$ is some positive constant.
\end{lm}

\begin{proof}
Since $b(x, i)$ and $\sigma(x, i)$ are Lipschitzian in $x$ uniformly in $N_k$, by standard arguments (e.g., \cite[Theorem 3.23]{MY}),
we can show that
$$\E|x(t\wedge\iota_k)-x(s\wedge\iota_k)|^6<\tilde C_k (t-s)^3, \forall 0\leq s\leq t\leq T.$$
Using the Kolmogorov-Centsov theorem,
we
obtain the first inequality.
The details are similar to the proof of Lemma \ref{lm3.1} in the Appendix.
The second claim
 is proved in the same manner as that of \cite[Lemma 2.14]{YZ}.
\end{proof}

{
Having Lemmas \ref{lm3.1} and \ref{lm3.2}, we are ready to
use the change of measure \eqref{e3.1} to prove
the Feller property of $(X_t,\alpha(t))$ under Assumption \ref{asp3.1}.
}

\begin{prop}\label{prop3.3}
Suppose that Assumption {\rm\ref{asp3.1}} is satisfied.
Let $f(\cdot,\cdot): \C\times\N\mapsto\R$ be continuous and bounded.
Then for any $T>0$,
$u_f(\phi,i)=\E f(X_T^{\phi,i},\alpha^{\phi,i}(T))$ is a continuous function in $\phi\in\C$.
\end{prop}

\begin{proof}
We suppose without loss of generality that $|f(\phi,i)|\leq 1\,\forall\,(\phi,i)\in\C\times\N$.
Fix $(\phi_0,i_0)\in\C\times\N$.
We show that for any $\Delta>0$, there exists $d^*=d^*(\Delta,\phi_0,i_0)>0$ such that
\begin{equation}\label{e3.18}
\big|\E f(X_T^{\phi,i_0},\alpha^{\phi,i_0}(T))-\E f(X_T^{\phi_0,i_0},\alpha^{\phi_0,i_0}(T))\big|\leq3\Delta\,\forall\,\|\phi-\phi_0\|<d^*
.\end{equation}
In view of Lemma \ref{lm3.1}, there are $m$, $n_m\in\N$, and $d_m>0$ such that
\begin{equation}\label{e3.19}
\PP\Big(\{\tau_{m+1}^{\phi,i_0}> T\}\cap\{\alpha^{\phi,i_0}(t)\in N_{n_m},\forall t\in[0,T] \}\Big)\geq 1-\Delta\,\forall \|\phi-\phi_0\|<\dfrac{d_m}2.
\end{equation}
Let $\eps=\eps(\Delta)>0$ (to be specified later).
Let $\hbar_k$ be as in  Lemma \ref{lm3.2}. Denote
$$\tilde\h=\Big\{\psi(\cdot)\in \C: \|\psi\|\leq R+1 \text{ and } \sup\limits_{t,s\in[-r,0],0<t-s<\hbar_{n_m}}\dfrac{|\psi(s)-\psi(t)|}{(s-t)^{0.25}}\leq 4\Big\}$$
and $\tilde \K=\{\phi_0\}\uplus\tilde\h$. 
By the compactness of $\tilde K$, there is  a $\tilde d_m>0$ such that
\begin{equation}\label{e3.20}
\|q_{ij}(\psi)-q_{ij}(\phi)\|<\eps, \ |f(\psi,i)-f(\phi, i)|<\eps
\end{equation} if $\phi\in\tilde \K, i, j\in N_{n_m}$ and $\|\psi-\phi\|<\tilde d_m$.
In view of Lemma \ref{lm3.2}, we can choose $\hat d_m>0$ such that
\begin{equation}\label{e3.20a}
\PP\Big\{\sup\limits_{t\in[0,T\wedge \iota_k]}\|Z^{\phi,i_0}_t-Z^{\phi_0,i_0}_t\|\leq \tilde d_m\Big\}<\eps \text{ if } \|\phi-\phi_0\|\leq\hat d_m.
\end{equation}
Let $A^{\phi}$ be the event $\{\tau^{\phi,i_0}_m\leq T<\tau^{\phi,i_0}_{m+1}, \iota_{n_m}>T\}$
and $l(T)$ be the number of jumps up to time $T$ of $\gamma^{i_0}(t)$.
It follows from Proposition \ref{prop3.2} that
\begin{equation}\label{e3.21}
\begin{aligned}
\E\big[ f&(X^{\phi,i_0}_T,\alpha^{\phi,i_0}(T))\1_{A^{\phi}}\big]-\E\big[ f(X^{\phi_0,i_0}_T,\alpha^{\phi_0,i_0}(T))\1_{A^{\phi_0}}\big]\\
=&\exp(T)\E\bigg[ \1_{\{l(T)<m+1,\iota_{n_m}>T\}}
\Big[g(Z^{\phi,i_0}(\cdot),\gamma^{i_0}(\cdot))-g(Z^{\phi_0,i_0}(\cdot),\gamma^{i_0}(\cdot))\Big]\bigg],
\end{aligned}
\end{equation}
where
$$g(Z^{\phi,i_0}(\cdot),\gamma^{i_0}(\cdot))=f(Z^{\phi,i_0}_T,\gamma^{i_0}(T))\prod_{k=1}^{l(T)}\dfrac{q_{\gamma^{i_0}_k\gamma^{i_0}_{k+1}}(Z^{\phi,i_0}_{(k+1)})}{\rho_{\gamma^{i_0}_k\gamma^{i_0}_{k+1}}}
\exp\Big\{-\int_0^Tq_{\gamma^{i_0}(s)}(Z^{\phi,i_0}_s)ds\Big\}.$$
Let $D_m^\phi$ be the event
$$D_m^\phi:=\Big\{\sup\limits_{t\in[0,T\wedge \iota_k]}\|Z^{\phi,i_0}_t-Z^{\phi_0,i_0}_t\|\leq \tilde d_m\Big\}\cap\Big\{\sup\limits_{t,s\in[0,T\wedge \iota_k],0<t-s<\hbar_{n_m}}\dfrac{|Z^{\phi_0,i_0}(s)-Z^{\phi_0,i_0}(t)|}{(s-t)^{0.25}}\leq 4\Big\}.$$
Using  \eqref{e3.21} and the estimates in \cite[Lemma 2.17]{YZ}, we obtain for $l\geq1$,
$$
\begin{aligned}
\Big|\E\big[& f(X^{\phi,i_0}_T,\alpha^{\phi,i_0}(T))\1_{A^{\phi}_l}\big]-\E\big[ f(X^{\phi_0,i_0}_T,\alpha^{\phi_0,i_0}(T))\1_{A^{\phi_0}_l}\big]\Big|\\
\leq& K \E\Big[\1_{\{\theta^{i_0}_l\leq T<\theta_{l+1}^{i_0}, \iota_{n_m}>T\}}\times\sup_{i\in N_{n_m}}|f(Z^{\phi,i_0}_T,i)-f(Z^{\phi_0,i_0}_T,i)|\Big]\\
&+K\E\Big[\1_{\{\theta^{i_0}_l\leq T<\theta_{l+1}^{i_0}, \iota_{n_m}>T\}}\times\sup_{t\in[0,T], i,j\in N_{n_m}}|q_{ij}(Z^{\phi,i_0}_t)-q_{ij}(Z^{\phi_0,i_0}_t)|\Big]\\
\leq& K \E\Big[\1_{\{\theta^{i_0}_l\leq T<\theta_{l+1}^{i_0}, \iota_{n_m}>T\}}\1_{D_m^\phi}\times\sup_{i\in N_{n_m}}|f(Z^{\phi,i_0}_T,i)-f(Z^{\phi_0,i_0}_T,i)|\Big]\\
&+K \E\Big[\1_{\{\theta^{i_0}_l\leq T<\theta_{l+1}^{i_0}, \iota_{n_m}>T\}}\1_{D_m^\phi}\times\sup_{t\in[0,T], i,j\in N_{n_m}}|q_{ij}(Z^{\phi,i_0}_t)-q_{ij}(Z^{\phi_0,i_0}_t)|\Big]\\
&+2K(M+1)\PP(\Omega\setminus D_m^\phi),
\end{aligned}
$$
where $K$ is a constant depending only on $T, m, n_m.$

Note that if $\omega\in\{\theta_l^{i_0}\leq T<\theta_{l+1}^{i_0}, \iota_{n_m}>T\}\cap D_m^\phi$, then $Z^{\phi_0,i_0}_t\in\tilde\K$ and $\|Z^{\phi,i_0}_t-Z^{\phi_0,i_0}_t\|\leq\tilde d_m\,\forall t\in[0,T]$ which implies in view of \eqref{e3.20}  that
$$\sup_{i\in N_{n_m}}|f(Z^{\phi,i_0}_T,i)-f(Z^{\phi_0,i_0}_T,i)|+\sup_{t\in[0,T], i,j\in N_{n_m}}|q_{ij}(Z^{\phi,i_0}_t)-q_{ij}(Z^{\phi_0,i_0}_t)|<2\eps.$$
On the other hand, Lemma \ref{lm3.2} and \eqref{e3.20a} imply that $$\PP(\Omega\setminus D_m^\phi)\leq 3\eps
\ \hbox{ if } \ \|\phi-\phi_0\|\leq\hat d_m.$$
Hence for $\|\phi-\phi_0\|\leq\hat d_m$, we have  that
\begin{equation}\label{e4.24}
\Big|\E\big[ f(X^{\phi,i_0}_T,\alpha^{\phi,i_0}(T))\1_{A^{\phi}}\big]-\E\big[ f(X^{\phi_0,i_0}_T,\alpha^{\phi_0,i_0}(T))\1_{A^{\phi_0}}\big]\Big|\leq 2K(4+3M)\eps,
\end{equation}
Note that $$\PP\big(\Omega\setminus A^{\phi}\big)=\PP\big(\{\tau^{\phi,i_0}_{m+1}<T\}\cup\{\alpha^{\phi,i_0}(t)\notin N_{n_m} \text{ for some } t\in[0,T]\}\big)<\Delta,$$
which implies
\begin{equation}\label{e4.25}
\Big|\E\big[ f(X^{\phi,i_0}_T,\alpha^{\phi,i_0}(T))\1_{\Omega\setminus A^{\phi}}\big]-\E\big[ f(X^{\phi_0,i_0}_T,\alpha^{\phi_0,i_0}(T))\1_{\Omega\setminus A^{\phi}}\big]\Big|\leq2\Delta.
\end{equation}
Choosing $\eps=\dfrac{\Delta}{2K(4+3M)}$, we have from \eqref{e4.24} and \eqref{e4.25} that
$$\Big|\E\big[ f(X^{\phi,i_0}_T,\alpha^{\phi,i_0}(T))\big]-\E\big[ f(X^{\phi_0,i_0}_T,\alpha^{\phi_0,i_0}(T))\big]\Big|\leq3\Delta$$
if $\|\phi-\phi_0\|<d^*:=\dfrac{d_m}2\wedge \hat d_m$.
\end{proof}
With the above technical preparations,
we are now in a position to prove the main theorem of this section.
{ By using truncation arguments, we can obtain the Feller property of $(X_t,\alpha(t))$
even if
$b(\cdot, i), \sigma(\cdot,i)$ do not vanish outside a bounded region
and
$q_{i}(\phi)$ is not bounded.
The precise condition is given below.}

\begin{thm}\label{thm3.2}
Let either Assumption \ref{asp2.1} or Assumption {\ref{asp2.2}} be satisfied.
Assume further that $q_{ij}(\cdot)$ is a continuous function for any $i, j\in\N$.
Then the solution to \eqref{e2.3} has the Feller property.
\end{thm}

\begin{proof}
Let $f(\cdot,\cdot):\C\times\N\mapsto\R$ be a continuous function with $|f(\phi, i)|\leq 1\,\forall\,(\phi,i)\in \C\times\N$.
Fix $R>0, T>0$. Suppose that $\|\phi_0\|<R$.
Under the hypotheses of Theorem \ref{thm2.1}, or Theorem \ref{thm2.2}, or Theorem \ref{thm2.3},
it is shown in the proofs of those theorems that
for any $\eps>0$, there exists an $\tilde R>0$ such that
\begin{equation}\label{e.xtruncate}
\PP\{\|X^{\phi,i_0}_t\|\leq\tilde R\}>1-\dfrac\eps8\,\forall \|\phi\|\leq R+1.
\end{equation}
Let $\Phi(x): \R^n\mapsto\R$ be a twice continuously differentiable satisfying
$\Phi(x)=1$ if $|x|\leq\tilde R$
and $\Phi(x)=0$ if $|x|\geq \tilde R+1$.
Let $(\tilde X^{\phi,i_0}_t,\tilde \alpha^{\phi,i_0}(t))$ be the solution with initial data $(\phi,i_0)$ to
\begin{equation}\label{e3.22}
\begin{cases}
d\tilde X(t)=\Phi(\tilde X(t))b(\tilde X(t), \tilde \alpha(t))dt+\Phi(\tilde X(t))\sigma(\tilde X(t),\tilde \alpha(t))dW(t) \\
d\tilde \alpha(t)=\disp\int_{\R}h(\tilde X_t,\tilde \alpha(t-), z)\p(dt, dz).\end{cases}
\end{equation}
Then $(\tilde X^{\phi,i_0}(t),\tilde\alpha^{\phi,i_0}(t))=(X^{\phi,i_0}(t),\alpha^{\phi,i_0}(t))$ up to the time that $\|X^{\phi,i_0}_t\| >\tilde R$, which combined with \eqref{e.xtruncate}
implies
$$
\PP\{\tilde\Omega_{\phi,i_0}\}>1-\dfrac\eps8, \ \forall\,\|\phi\|<R
$$
where
$\tilde\Omega_{\phi,i_0}:=\{\tilde X^{\phi,i_0}_T=X^{\phi,i_0}_T, \tilde\alpha^{\phi,i_0}(T)=\alpha^{\phi,i_0}(T)\}$.
As a result, if $\|\phi\|<R$, we have
\begin{equation}\label{e3.23}
\begin{aligned}
\Big|\E f&( X^{\phi,i_0}_T,\alpha^{\phi,i_0}(T))-\E f( \tilde X^{\phi,i_0}_T,\tilde\alpha^{\phi,i_0}(T))\Big|\\
\leq&\E\left[\1_{\tilde\Omega^c_{\phi,i_0}}\Big|f( X^{\phi,i_0}_T,\alpha^{\phi,i_0}(T))-f( \tilde X^{\phi,i_0}_T,\tilde\alpha^{\phi,i_0}(T))\Big|\right]\,\, (\text{with }\tilde\Omega^c_{\phi,i_0}=\Omega\setminus \tilde\Omega_{\phi,i_0})\\
\leq&2\PP\left(\1_{\tilde\Omega^c_{\phi,i_0}}\right)\leq 2\dfrac\eps8=\dfrac\eps4.
\end{aligned}
\end{equation}

It follows from Proposition \ref{prop3.3} that there exists a $\delta\in(0,1)$ such that if $\|\phi-\phi_0\|<\delta$, we have
\begin{equation}\label{e3.24}
\Big|\E f(\tilde X^{\phi,i_0}_T,\tilde\alpha^{\phi,i_0}(T))-\E f(\tilde X^{\phi_0,i_0}_T,\tilde\alpha^{\phi_0,i_0}(T))\Big|<\dfrac\eps2.
\end{equation}
Since $\|f\|\leq 1$, we can easily obtain from \eqref{e3.23} and \eqref{e3.24} that
$$
\begin{aligned}
\Big|\E& f( X^{\phi,i_0}_T,\alpha^{\phi,i_0}(T))-\E f( X^{\phi_0,i_0}_T,\alpha^{\phi_0,i_0}(T))\Big|\\
\leq&\Big|\E f(\tilde X^{\phi,i_0}_T,\tilde\alpha^{\phi,i_0}(T))-\E f( \tilde X^{\phi_0,i_0}_T,\tilde\alpha^{\phi_0,i_0}(T))\Big|\\
&+\Big|\E f( X^{\phi,i_0}_T,\alpha^{\phi,i_0}(T))-\E f( \tilde X^{\phi,i_0}_T,\tilde\alpha^{\phi,i_0}(T))\Big|\\
&+\Big|\E f( X^{\phi_0,i_0}_T,\alpha^{\phi_0,i_0}(T))-\E f(\tilde X^{\phi_0,i_0}_T,\tilde\alpha^{\phi_0,i_0}(T))\Big|\\
<&\dfrac\eps2+\dfrac\eps4+\dfrac\eps4=\eps,\,\text{ if }\, \|\phi-\phi_0\|<\delta.
\end{aligned}
$$
The proof of the theorem is complete.
\end{proof}

\section{Feller Property of Hybrid Diffusion without Past Dependence}\label{sec:fel}
Now, we suppose that the
$q_{ij}, i,j\in\N$  associated with $\alpha(t)$ depend only the current state of $X(t)$.
To be more precise $q_{ij}(\cdot)$ is a function from $\R^n$ to $\R$ for each $(i,j)\in\N\times \N$.
As a special case of the hybrid diffusion with past-dependent switching, we obtain the following theorem.

\begin{thm}\label{thm5.1}
Assume that $q_{ij}(\cdot)$ is a continuous function for any $i, j\in\N$. Assume further that one of the following conditions is satisfied:
\begin{enumerate}[{\rm(A)}]
\item Assumption {\rm\ref{asp2.3}} and $q_{i}(y)=\sum_{j\ne i}q_{ij}(y)$ is bounded uniformly in $(i, y)\in\N\times\R^n$.
\item Assumption {\rm\ref{asp2.4}} and $q_{i}(y)$ is bounded uniformly in $(i,  y)\in\N\times K$ for each compact subset $K$ of $\R^n$.
\end{enumerate}
Then the unique solution to \eqref{e2.3} is a Markov process having the Feller property.
\end{thm}

\begin{rem}\label{rem:con-f}{\rm
If for each discrete state $i \in \N$,
the diffusion $Y^{(i)}(t)$, which is
the solution process to
\begin{equation}\label{e5.1}
d Y^{(i)}(t)=b(Y^{(i)}(t),i)dt+\sigma(Y^{(i)}(t),i)dW(t)
\end{equation}
has the strong Feller property, we do not need the continuity of $q_{ij}(\cdot)$ to get the Feller property of $(X(t),\alpha(t))$.
In fact, we will obtain a stronger result, namely, the strong Feller property.
The condition for the strong Feller property of $Y^{(i)}(t)$
is essentially the ellipticity of $A(x, i)$ or the H\"ormander condition for hyperellipticity (see e.g., \cite{DN, SV}).
}\end{rem}

\begin{thm}\label{thm5.2}
Assume that $q_{ij}(\cdot)$ is measurable for any $i, j\in\N$ and either  {\rm(A)} or {\rm(B)} in Theorem {\rm\ref{thm5.1}} holds.
If for each $i\in\N$, the solution process $Y^{(i)}(t)$ to \eqref{e5.1} has the strong Feller property, then the unique solution to \eqref{e2.3} has the
 strong Feller property, that is, for each bounded measurable function $g(y,i):\R^n\times\N\to\R$,
the function $(x,i)\to \E g(X^{x,i}(T),\alpha^{x,i}(T))$ is continuous for each $T>0$.
\end{thm}

\begin{proof}
We assume without loss of generality that $|g(z,i)|\leq 1\,\forall z\in\R^n,i\in\N$.
Let $Y^{y,i}(t)$ be the solution with initial data $y$ to \eqref{e5.1}. Fix $(x,i)\in\R^n\times\N$ and $\eps>0$.
Under the assumption $(A)$ or $(B)$, we can show that for each $x\in\R^n$, there is a  $K>0$ satisfying
\begin{equation}\label{e5.2}
\PP((\Omega_{\eps}^{y,i})^c)<\dfrac{\eps}8\,\forall y\in B(x,1):=\{z: |x-z|<1\},
\end{equation}
where $\Omega_{\eps}^{y,i}=\{|Y^{y,i}(t)|\leq K\,\forall t\in[0,1]\}$, $(\Omega_{\eps}^{y,i})^c=\Omega\setminus \Omega_{\eps}^{y,i}$.
Let $M=\sup_{i\in\N, |z|\leq K}q_{i}(z)$ and $t_0\in(0,1)$ satisfying
$1-\exp\{-Mt_0\}<\dfrac{\eps}{16}$. It follows from \eqref{e5.2} and (i) of Lemma \ref{lm3.0} that
\begin{equation}\label{e5.3}
\PP\{\tau^{y,i}>t_0\}>1-\dfrac{3\eps}{16}\,\forall y\in B(x,1),
\end{equation}
where $$\tau^{y,i}:=\inf\{t>0:\alpha^{y,i}(t)\ne i\}=\inf\Big\{t>0: \int_0^t\int_\R h(Y^{y,i}(s),i, u)\p(ds,du)\ne 0\Big\}.$$
Denote
$\Phi(y,i):=\E g\big(X^{y,i}(T-t_0),\alpha^{y,i}(T-t_0)\big)$.
The condition $|g(y, i)|\leq 1$ implies  $|\Phi(y,i)| \le 1$ for all $y\in\R^n, i\in\N$.
By the strong Feller property of $Y^{(i)}(t)$,
there is a $\delta>0$ such that
\begin{equation}\label{e5.4}
|\E \Phi(Y^{y,i}(t_0),i)-\E\Phi(Y^{x,i}(t_0),i)|\leq\dfrac\eps4\,\forall y\in B(x,\delta).
\end{equation}
By the strong Markov property of $X(t)$, we have
\begin{equation}\label{e5.5}
\begin{aligned}
\E  g\big(&X^{y,i}(T),\alpha^{y,i}(T)\big)=\E \Phi(X^{y,i}(t_0),\alpha^{y,i}(t_0))\\
&=\E \big[\1_{\{\tau^{y,i}>t_0\}}\Phi(X^{y,i}(t_0),\alpha^{y,i}(t_0))]+\E \big[\1_{\{\tau^{y,i}\leq t_0\}}\Phi(X^{y,i}(t_0),\alpha^{y,i}(t_0))].
\end{aligned}
\end{equation}
Applying (i) of Lemma \ref{lm3.0}, we obtain
\begin{equation}\label{e5.6}
\begin{aligned}
\E \big[&\1_{\{\tau^{y,i}>t_0\}}\Phi(X^{y,i}(t_0),\alpha^{y,i}(t_0))]\\
=&\E \big[\Phi(Y^{y,i}(t_0),\alpha^{y,i}(t_0))\exp\{-\int_0^{t_0}q_i(Y^{y,i}(s))ds\}\big]\\
=&\E \big[\Phi(Y^{y,i}(t_0),\alpha^{y,i}(t_0))\big]+\E \Big[\1_{(\Omega^{y,i}_\eps)^c}\Phi(Y^{y,i}(t_0),\alpha^{y,i}(t_0))\Big(\exp\{-\int_0^{t_0}q_i(Y^{y,i}(s))ds\}-1\Big)\Big]\\
&+\E \Big[\1_{\Omega_\eps^{y,i}}\Phi(Y^{y,i}(t_0),\alpha^{y,i}(t_0))\Big(\exp\{-\int_0^{t_0}q_i(Y^{y,i}(s))ds\}-1\Big)\Big].
\end{aligned}
\end{equation}
Note that if $|g(z,i)|\leq 1\,\forall z\in\R^n,i\in\N$ then $|\Phi(z,i)|\leq 1\,\forall z\in\R^n,i\in\N$.
We have the following estimates for $y\in B(x,\delta)$ using \eqref{e5.2}, \eqref{e5.3}, \eqref{e5.4}, and the fact that $|g(z,i)|, |\Phi(z,i)|\leq 1\,\forall z\in\R^n,i\in\N$.
\begin{equation}\label{e5.8}
\begin{aligned}
\Big|\E \big[\1_{(\Omega^{y,i}_\eps)^c}\Phi(Y^{y,i}(t_0),\alpha^{y,i}(t_0))\big(\exp\{-\int_0^{t_0}q_i(Y^{y,i}(s))ds\}-1\big)\big]\Big|
\leq\PP((\Omega^{y,i}_\eps)^c)\leq\dfrac{\eps}8,
\end{aligned}
\end{equation}
\begin{equation}\label{e5.9}
\Big|\E \Big[\1_{\Omega^{y,i}_\eps}\Phi(Y^{y,i}(t_0),\alpha^{y,i}(t_0))\Big(\exp\{-\int_0^{t_0}q_i(Y^{y,i}(s))ds\}-1\Big)\Big]\Big|\leq 1-\exp(-Mt_0)\leq \dfrac{\eps}{16},
\end{equation}
\begin{equation}\label{e5.10}
\E \big[\1_{\{\tau^{y,i}\leq t_0\}}\Phi(X^{y,i}(t_0),\alpha^{y,i}(t_0))]\Big|\leq\PP\{\tau^{y,i}\leq t_0\}\leq \dfrac{3\eps}{16}.
\end{equation}
Applying estimates \eqref{e5.4}, \eqref{e5.8}, \eqref{e5.9}, and \eqref{e5.10} to \eqref{e5.5} and \eqref{e5.6},
we have
$$\Big|\E \Phi(X^{y,i}(T),\alpha^{y,i}(T))-\E \Phi(X^{x,i}(T),\alpha^{x,i}(T))\Big|\leq\eps, \ \forall y\in B(x,\delta).$$
The proof is complete.
\end{proof}

\begin{rem}{\rm
Sufficient conditions for the strong Feller property of $(X(t), \alpha(t))$, in which the rates of switching $q_{ij}$ for $i,j\in\N$ depend only on the
current continuous state  $X(t)$, was obtained in Shao \cite{SJ}.
However, the conditions there are  restrictive.
To obtain the strong Feller property, it is assumed in
  \cite{SJ} that $q_{ij}(x), b(x, i)$ and $\sigma (x,i)$ are Lipschitz in $x$ uniformly in $i\in\N$.
  The ellipticity of $A(x,i)$ is also assumed
  to be uniform in $(x,i)\in\R^n\times\N$.
Moreover, it is assumed that $q_{ij}(x)=0$ if $|i-j|<k$ for some constant $k$.
It can be seen that our conditions in this paper are much more relaxed compared with the aforementioned reference.
}\end{rem}

\section{Further Remarks}\label{sec:rem}
This paper has been devoted to modeling and analysis of switching diffusions in which past-dependent
switching processes and countable switching sets are considered.  Many problems remain open.
Based on our results,  one may consider
 such properties as  recurrence, ergodicity, and stability.
Future work may also be directed to the study of switching diffusions in which the drift and diffusions are also past dependent.
It is important to work out all the details of the control problems presented in the previous sections, which may  open up a new avenue for
investigation of a wide range of control and optimization problems involving switching diffusions that are treated in this paper.

\appendix
\section{Appendix}\label{sec:sto}
This section is devoted to the proofs of some technical results.
To simplify the notation, we denote by $\PP_{\phi,i}$ the probability measure conditional on the initial value $(\phi,i)$,
that is, for any $t>0$,
$$\PP_{\phi,i}\{(X_t,\alpha(t))\in\cdot\}=\PP\{(X_t^{\phi,i},\alpha^{\phi,i}(t))\in\cdot\},$$
$$\PP_{\phi,i}\{(Y_t,\beta(t))\in\cdot\}=\PP\{(Y_t^{\phi,i},\beta^{\phi,i}(t))\in\cdot\},$$
and
$$\PP_{\phi,i}\{(Z_t,\gamma(t))\in\cdot\}=\PP\{(Z_t^{\phi,i},\gamma^{\phi,i}(t))\in\cdot\}.$$
Let $\E_{\phi,i}$ be the expectation associated with $\PP_{\phi,i}$.
First, we prove the following result.

\begin{lm}
Let either Assumption {\rm\ref{asp2.1}} or Assumption {\rm\ref{asp2.3}} combined with {\rm (ii)} of Assumption {\rm\ref{asp2.1}}
be satisfied.
Assume further
that $q_{i}(\cdot),q_{ij}(\cdot)$ are continuous functions in $\C$ for each $i,j\in\N$.
Then the solution $(X_t,\alpha(t))$ to \eqref{e2.3} satisfies \eqref{eq:sde} and \eqref{eq:tran}.
\end{lm}

\begin{proof}
It is clear that the solution $(X_t,\alpha(t))$ to \eqref{e2.3} satisfies \eqref{eq:sde}.
Fix $\phi\in\C,$ $i, j\in\N, i\ne j$. Applying the generalized It\^o formula
to the function $V(\psi,k)=0$ if $k\ne j$ and $V(\psi, j)=1$
we have
$$\PP_{\phi,i}\{\alpha(\Delta)=j\}=\E_{\phi,i}V(X_\Delta,\alpha(\Delta)
=\E_{\phi,i}\int_0^\Delta q_{\alpha(t),j}(X_t)dt,\,\,\, \text{ for } \Delta>0,$$
{\color{blue}where $q_{ii}(\phi):=-q_i(\phi)=-\sum_{j\ne i}q_{ij}(\phi)$.}
\footnote{ $q_{ii}(\phi)$ is not defined in the journal version (SICON) of this paper} Since $\alpha(t)$ is cadlag and $X(t)$ is continuous,
 $\lim_{t\to0^+}\alpha(t)=i$ and $\lim_{t\to0^+}X_t=\phi$ $\PP_{\phi,i}$-a.s.
In light of the continuity of $q_{ij}(\cdot)$ we obtain
$\lim_{t\to0^+}q_{\alpha(t),j}(X_t)=q_{ij}(\phi)\,\,\PP_{\phi,i}-\text{a.s.}$
which implies that
$$\lim_{\Delta\to0^+}\dfrac1\Delta\int_0^\Delta q_{\alpha(t),j}(X_t)dt=q_{ij}(\phi)\,\,\PP_{\phi,i}-\text{a.s.}$$
Since $q_{ij}(\cdot)$ is uniformly bounded, so is $ \disp\dfrac1\Delta\int_0^\Delta q_{\alpha(t),j}(X_t)dt$.
By virtue of the Lebegue dominated convergence theorem, we have
\begin{equation}\label{appendix.1}
\lim_{\Delta\to0^+}\dfrac{\PP_{\phi,i}\{\alpha(\Delta)=j\}}\delta
=\lim_{\Delta\to0^+}\E_{\phi,i}\left(\dfrac1\Delta\int_0^\Delta q_{\alpha(t),j}(X_t)dt\right)=q_{ij}(\phi).
\end{equation}
In the same manner, applying the generalized It\^o formula
to the function $V(\psi,k)=1$ if $k\ne i$ and $V(\psi, i)=0$,
we obtain that
\begin{equation}\label{appendix.2}
\lim_{\Delta\to0^+}\dfrac{1-\PP_{\phi,i}\{\alpha(\Delta)=i\}}\delta
=q_{i}(\phi).
\end{equation}
The proof is complete by
noting that \eqref{eq:tran} follows
from \eqref{appendix.1} and \eqref{appendix.2} and the Markov property of $(X(t),\alpha(t))$.
\end{proof}
Next, we provide the proofs of some results in Section 4.

\begin{proof}[Proof of Lemma \ref{lm3.0}]
To prove  claim (i), we apply the generalized It\^o formula to $V(j)=1$ if $j=i$, and $V(j)=0$ if $j\ne i$.
We have
$$V(\beta(\lambda_1\wedge t))=-\int_0^{\lambda_1\wedge t}q_{i}(Y_s)ds+\int_0^{\lambda_1\wedge t}\int_\R \Big(V\big(i+h(Y_t, i,z)\big)-1)\Big)\mu(ds,dz).$$
Since $W(\cdot)$ is independent of the Poisson random measure,
taking the conditional expectation with respect to $\F^W_T$ yields
$$
\begin{aligned}
\E_{\phi,i}\big[ \1_{\{\lambda_1> t\}}\big|\F^W_T\big]=&\E_{\phi,i}\big[ V(\lambda_1\wedge t)\big|\F^W_T\big]=-\E_{\phi,i}\big[\int_0^{\lambda_1\wedge t}q_{i}(Y_s)ds\big|\F^W_T\big]+1\\
=&-\E_{\phi,i}\big[\int_0^tq_{i}(Y_s)ds\1_{\{\lambda_1> s\}}\big|\F^W_T\big]+1\\
=&- \int_0^tq_{i}(Y_s)\E_{\phi,i}\big[\1_{\{\lambda_1> s\}}\big|\F^W_T\big]ds+1.
\end{aligned}
$$
Hence,
$$\dfrac{d}{dt}\E_{\phi,i}\big[ \1_{\{\lambda_1> t\}}\big|\F^W_T\big]=-q_{i}(Y_t)\E_{\phi,i}\big[ \1_{\{\lambda_1>t\}}\big|\F^W_T\big].$$
Since $\E_{\phi,i}\big[ \1_{\{\lambda_1> 0\}}\big|\F^W_T\big]=1$, we obtain
\begin{equation}\label{e3.2}
\PP_{\phi,i}\big(\{\lambda_1> t\}\big|\F^W_T\big)=\E_{\phi,i}\big[ \1_{\{\lambda_1> t\}}\big|\F^W_T\big]=\exp\Big(-\int_0^tq_{i}(Y_s)ds\Big)
.\end{equation}
Now we prove  claim (ii).
First, we
 try to find the distribution of $(\lambda_1, \beta_1)$ conditioned on $\F_T^W$ when $\lambda_1\in[0,T]$. Fix $j\ne i$ and let $f(t, k): [0,T]\times\N\to\N$ be any bounded measurable function satisfying
$f(t,k)=0$ if $k\ne j$.
Applying the generalized It\^o formula, we obtain
$$
\begin{aligned}
f(\lambda_1\wedge T, \beta(\lambda_1\wedge T))=&
\int_0^{\lambda_1\wedge T}q_{ij}(Y_t)f(t, j)dt\\
&+\int_0^{\lambda_1\wedge T}\int_\R \Big(f\big(s, i+h(Y_t, i,z)\big)-f(Y_t,i)\Big)\mu(ds,dz)
.\end{aligned}
$$
Since $W(\cdot)$ is independent of the Poisson random measure,
taking the conditional expectation with respect to $\F^W_T$, we have
$$
\begin{aligned}
\E_{\phi,i}\big[f(\lambda_1\wedge T, \beta(\lambda_1\wedge T))\big|\F^W_T\big]=&
\E_{\phi,i}\Big[\int_0^{\lambda_1\wedge T}q_{ij}(Y_t)f(t, j)dt\Big|\F^W_T\Big]\\
=&\E_{\phi,i}\Big[\int_0^{T}q_{ij}(Y_t)f(t, j)dt\1_{\{\lambda_1>t\}}dt\Big|\F^W_T\Big]\\
=&\int_0^{T}q_{ij}(Y_t)f(t, j)\E\big[\1_{\{\lambda_1>t\}}\big|\F^W_T\big]dt\\
=&\int_0^{T}q_{ij}(Y_t)f(t, j)\exp(-\int_0^tq_{i}(Y_s)ds)dt.
\end{aligned}
$$
As a result, for $t\in[0,T]$,
$$\PP_{\phi,i}\big\{\lambda_1\in dt, \beta(\lambda_1)=j\big|\F_T^W\big\}=q_{ij}(Y_t)\exp(-\int_0^tq_{i}(Y_s)ds)dt.$$
Thus,
$$
\begin{aligned}
\E_{\phi,i}\Big[g&(Y_{(1)},\lambda_1,\beta_1)\1_{\{\lambda_1\leq T\}}\Big|\F^W_T\Big]\\
=&\sum_{j=1, j\ne i}^\infty\int_0^Tg(Y_t,t,j)\PP_{\phi,i}\big\{\lambda_1\in dt, \beta(\lambda_1)=j\big|\F_T^W\big\}\\
=&\sum_{j=1, j\ne i}^\infty\int_0^Tg(Y_t,t,j)q_{ij}(Y_t)\exp(-\int_0^tq_{i}(Y_s)ds)dt
\end{aligned}
$$
as desired.
\end{proof}

\begin{proof}[Proof of Proposition \ref{prop3.2}]
First, we prove \eqref{e3.1} for the case $l=0$.
Since $(X_t,\alpha(t))=(Y_t,\lambda(t))$ up to the moment $\alpha_1=\lambda_1$,
we have
\begin{equation}\label{e3.5}
\begin{aligned}
\E_{\phi,i} \Big[f(&X_T,\alpha(T))\1_{\{\tau_1>T\}}\Big]=\E_{\phi,i} \Big[f(Y_T, i)\1_{\{\lambda_1>T\}}\Big]\\
=&\E_{\phi,i}\Big[\E_{\phi,i}\big(f(Y_T, i)\1_{\{\lambda_1>T\}}|\F^W_T\big)\Big]
=\E_{\phi,i}\Big[f(Y_T, i)\E_{\phi,i}\big(\1_{\{\lambda_1>T\}}|\F^W_T\big)\Big]\\
=&\E_{\phi,i}\Big[f(Y_T, i) \exp\Big(-\int_0^Tq_{i}(Y_s)ds\Big)\Big],
\end{aligned}
\end{equation}
where the last equality is
consequence of (i)
of Lemma \ref{lm3.0}.
Since $\gamma(\cdot)$ and $Y(\cdot)$ are independent, $Z_t=Y_t$ up to the moment $\theta_1$ and $\PP_{\phi,i}\{\theta_1>T\}=\exp(-T)$, we obtain
\begin{equation}\label{e3.6}
\begin{aligned}
\E_{\phi,i} \Big(f&(Z_T,\gamma(T))\1_{\{\theta_1>T\}} \exp\{-\int_0^Tq_{i}(Z_s)ds\}\Big)\\
=&\E_{\phi,i}\Big(f(Y_T, i)\1_{\{\theta_1>T\}} \exp\{-\int_0^Tq_{i}(Y_s)ds\}\Big)\\
=&\PP_{\phi,i}\{\theta_1>T\}\E_{\phi,i}\Big(f(Y_T, i) \exp\{-\int_0^Tq_{i}(Y_s)ds\}\Big)\\
=&\exp(-T)\E_{\phi,i}\Big(f(Y_T, i) \exp\{-\int_0^Tq_{i}(Y_s)ds\}\Big)
.\end{aligned}
\end{equation}
From \eqref{e3.5} and \eqref{e3.6}, we have for $t\in[0,T]$  that
\begin{equation}\label{e3.7}
\E f(X_T,\alpha(T))\1_{\{\tau_1>T\}}=\exp(T)\E \Big(f(Z_T,\gamma(T))\1_{\{\theta_1>T\}} \exp\{-\int_0^Tq_{i}(Z_s)ds\}\Big).
\end{equation}
We now prove \eqref{e3.1} for $l=1$.
Let $g(\phi, t, i)$, $\tilde g(\phi, t, i): \C\times[0,\infty)\times\N\to\R$ be bounded measurable functions and $g(\phi, t, i)=\tilde g(\phi, t, i)=0$ if $t>T$. It follows from (ii) of Lemma \ref{lm3.0} that
\begin{equation}\label{e3.8}
\begin{aligned}
\E_{\phi,i} g&(X_{(1)},\tau_1, \alpha_1)
=\E_{\phi,i} g(Y_{(1)},\lambda_1, \beta_1)\\
=&\sum_{i_1\ne i} \int_0^T \E_{\phi,i}\Big(g(Y_t, t, i_1)q_{ii_1}(Y_t)\exp(-\int_0^tq_{i}(Y_s)ds)\Big)dt.
\end{aligned}
\end{equation}
On the other hand, 
\begin{equation}\label{e3.9}
\begin{aligned}
\E_{\phi,i} \Big[\tilde g&(Z_{(1)},\theta_1, \gamma_1) \exp(-\int_0^{\theta_1}q_{i}(Z_s)ds)\Big]\\
=&\E_{\phi,i}\Big[\tilde g(Y_{(1)},\theta_1, \gamma_t) \exp(-\int_0^{\theta_1}q_{i}(Y_s)ds)\Big]\\
=&\sum_{i_1\ne i} \int_0^T \E_{\phi,i}\Big( \tilde g(Y_t, t, i_1)\exp(-\int_0^tq_{i}(Y_s)ds)\Big)\PP_{\phi,i}\{\theta_1\in dt, \gamma_1=i_1\}\\
=&\sum_{i_1\ne i} \int_0^T \E_{\phi,i}\Big(\tilde g(Y_t, t, i_1)\exp(-\int_0^tq_{i}(Y_s)ds)\Big)\rho_{ii_1}\exp(-t)dt.
\end{aligned}
\end{equation}
Substituting $\tilde g(\phi, t, i)=g(\phi, t, i) \exp(t)\times\dfrac{q_{ii}(\phi)}{\rho_{i i}}$ into \eqref{e3.9}, we have
\begin{equation}\label{e3.10}
\begin{aligned}
\E_{\phi,i}\Big[ g&(Z_{(1)},\theta_1, \gamma_1) \exp(\theta_1)\times\dfrac{q_{i\gamma_1}(Z_s)}{\rho_{i\gamma_1}} \exp(-\int_0^{\theta_1}q_{i}(Z_s)ds)\Big]\\
=&\sum_{i_1\ne i} \int_0^T \E_{\phi,i}\Big[ g(Y_t, t, i_1)\exp(t)\dfrac{q_{ii_1}(Y_t)}{\rho_{ii_1}}\exp(-\int_0^tq_{i}(Y_s)ds)\Big]\rho_{ii_1}\exp(-t)dt\\
=&\sum_{i_1\ne i} \int_0^T \E_{\phi,i}\Big[ g(Y_t, t, i_1)q_{ii_1}(Y_t)\exp(-\int_0^tq_{i}(Y_s)ds)\Big]dt.
\end{aligned}
\end{equation}
It follows from \eqref{e3.8} and \eqref{e3.10} that
\begin{equation}\label{e3.11}
\begin{aligned}
\PP_{\phi,i}\{&\tau_1\in dt, \alpha_1=i_1, X_{(1)}\in d\phi_1\}\\
=&\E_{\phi,i}\Big[\1_{\{\theta_1\in dt, \gamma_1=i_1, Z_{(1)}\in d\phi_1\}} \exp(t)\times\dfrac{q_{ii_i}(Z_t)}{\rho_{ii_1}} \exp(-\int_0^{t}q_{i}(Z_s)ds)\Big].
\end{aligned}
\end{equation}
We now use the strong Markov property of $(X_t,\alpha(t))$ and $(Z_t,\gamma(t))$,  \eqref{e3.11} as well as \eqref{e3.7} with $\phi,i, T$ replaced
by $\phi_1,i_1, T-t$, respectively;
\begin{equation}
\begin{aligned}
\E_{\phi,i} f&(X_T,\alpha(T))\1_{\{\tau_1\leq T<\tau_2\}}\1_{\{\alpha_1=i_1\}}\\
=&\int_0^T\int_\C\Big[\PP\{\tau_1\in dt, \alpha_1=i_1, X_{(1)}\in d\phi_1\}\times\E_{\phi_1,i_1} f(X_{T-t},\alpha(T-t))\1_{\{\tau_1>T-t\}}\Big]\\
=&\int_0^T\int_\C\Big[\E_{\phi,i}\Big(\1_{\{\theta_1\in dt, \gamma_1=i_1, Z_{(1)}\in d\phi_1\}} \exp(t)\times\dfrac{q_{ii_i}(Z_s)}{\rho_{ii_1}} \exp(-\int_0^{t}q_{i}(Z_s)ds)\Big)\\
&\qquad\times\exp(T-t)\E_{\phi_1,i_1} f(Z_{T-t},\gamma(T-t))\1_{\{\theta_1>T-t\}}\exp\{-\int_0^{T-t}q_{i_1}(Z_s)ds\}\Big]\\
=&\exp(T)\E_{\phi,i}\Big[ f(Z_T,\gamma(T))\1_{\{\theta_1\leq T<\theta_2\}}\1_{\{\gamma_1=i_1\}}\exp\Big(-\int_{\theta_1}^Tq_{i_1}(Z_s)ds\Big)\\
&\qquad\qquad\qquad \dfrac{q_{ii_1}(Z_{(1)})}{\rho_{ii_1}}\exp\Big(-\int_0^{\theta_1}q_{i}(Z_s)ds\Big)\Big].
\end{aligned}
\end{equation}
We have already proved \eqref{e3.1} for $l=0, 1$. Using the same argument, the induction,  and the strong Markov property of $(X_t,\alpha(t))$ and $(Z_t,\gamma(t))$, we can  obtain \eqref{e3.1} for any $l\in\N$.
\end{proof}
\begin{proof}[The proof of Lemma \ref{lm3.1}]
 By \eqref{e2.9}, we can find $m\in\N$ such that
\begin{equation}\label{e3.12}
\PP_{\phi,i_0}\big\{\tau_{m+1}<T\big\}<\dfrac{\Delta}2, \ \forall\,(\phi,i)\in\C\times\N.
\end{equation}
Now, let $\eps=\eps(\Delta)>0$
(to be specified
later).
In view of \cite[Theorem 4.3, p. 61]{XM}, for each $i\in\N$, there is a constant $C_i$ such that
\begin{equation}\label{e3.13}
\E_{\phi,i_0}|Y(t)-Y(s)|^6\leq C_i|t-s|^3\,\forall t,s\in[0,T],\,\forall\,\|\phi\|\leq R+1.
\end{equation}
By the Kolmogorov-Centsov theorem (see \cite[Theorem 2.8]{KS}), there is a positive random variable $h^\phi_i(\omega)$ such that
$$\PP_{\phi,i_0}\Big\{\sup\limits_{t,s\in[0,T],0<t-s<h^\phi_i(\omega)}\dfrac{|Y^{\phi, i}(t)-Y(s)|}{(t-s)^{0.25}}\leq 4\Big\}=1.$$
Since $C_i$ in \eqref{e3.13} does not depend on $\phi\in\{\psi: \|\psi\|\leq R+1\}$, it can be seen from the proof of the Kolmogorov-Centsov theorem that for any $\eps>0$, there is a constant $h_i>0$ satisfying
\begin{equation}\label{e3.12a}
\PP_{\phi,i_0}\Big\{\sup\limits_{t,s\in[0,T],0<t-s<h_i}\dfrac{|Y(t)-Y(s)|}{(s-t)^{0.25}}\leq 4\Big\}>1-\eps ,\ \forall\,\|\phi\|\leq R+1.
\end{equation}
Without loss of generality, we can choose $h_{i+1}<h_i,\forall i\in\N$.
Let
$$\h_{i, T}=\Big\{\psi(\cdot)\in \C([0,T],\R): \|\psi\|\leq R+1 \text{ and }\sup\limits_{t,s\in[0,T],0<t-s<h_i}\dfrac{|\psi(s)-\psi(t)|}{(s-t)^{0.25}}\leq 4\Big\},$$
and
$$\h_i=\Big\{\psi(\cdot)\in \C: \|\psi\|\leq R+1 \text{ and } \sup\limits_{t,s\in[-r,0],0<t-s<h_i}\dfrac{|\psi(s)-\psi(t)|}{(s-t)^{0.25}}\leq 4\Big\}.$$
Hence $\h_{i+1, T}\supset\h_{i, T}$ and $\h_{i+1}\supset\h_i$.
For $d>0$ and a compact set $\K\subset\C$, we define
$$\K_{d}:=\{\psi\in\C: \exists\phi\in\K\text{ such that } \|\psi-\phi\|<d\}.$$
Define $\K^0=\{\psi(\cdot)=\phi_0(\cdot)+c: c\in\R^n, |c|\leq 1 \}$, which is compact,
and $\K^1=\K^0\uplus\h_{i_0}.$
For each $\phi\in\K^1$, there is $n_{\phi,i_0}>i_0$  such that
$$\sum_{k=n_{\phi,i_0}+1}^\infty q_{i_0, k}(\phi)= q_{i_0}(\phi)-\sum_{k=1,k\ne i_0}^{n_{\phi,i_0}}q_{i_0k}(\phi)<\dfrac{\eps}2.$$
By the continuous of $q_{i_0}$ and $q_{i_0k}(\phi)$, there is a $d_{\phi,i_0}>0$ such that
$$\sum_{k=n(\phi)+1}^\infty q_{i_0, k}(\phi')=q_{i_0}(\phi')-\sum_{k=1,k\ne i_0}^{n_{\phi,i_0}}q_{i_0k}(\phi')<\eps\,\forall\, \|\phi'-\phi\|<d_{\phi,i_0}.$$
Since $\K^1$ is compact, there exist $n_1>0$ and $d_1>0$ such that
$$\sum_{k=n_1+1}^\infty q_{i_0, k}(\phi)<\eps\,\forall \phi\in\K^1_{d_1}.$$
Define $\K^2=\K^1\uplus\h_{n_1}.$
Using the compactness of $\K^2$, there exist $n_2>n_1$ and $d_2\in(0,d_1]$ such that
$$\sum_{k=n_2+1}^\infty q_{i, k}(\phi)<\eps\,\forall i\in N_{n_1},\phi\in\K^2_{d_2}.$$
Continuing this way, for
$\K^{m}=\K^{m-1}\uplus\h_{n_{m-1}},$ there exists $n_m>n_{m-1}$ and $d_m\in(0,d_{m-1}]$ such that
$$\sum_{k=n_{m}+1}^\infty q_{i, k}(\phi)<\eps\,\forall i\in N_{n_{m-1}},\phi\in\K^m_{d_m}.$$
Set $\K^{\phi,1}=\{\phi\}\uplus\h_{i_0}$ and $\K^{\phi,k}=\K^{\phi,k-1}\uplus\h_{n_{k-1}}$ for $\phi\in\C$ and $k=2,\dots,m$.
It is not difficult to verify that
\begin{equation}\label{e3.14}
\K^{\phi,k}\subset\K^k_{d_k}\,\forall\,k=1,\dots,m, \text{ for }\|\phi-\phi_0\|<\dfrac{d_m}2.
\end{equation}
Denote by $\{Y(\cdot)\in \h_{n_0,T}\}$ the event $\{t\in[0,T]\mapsto Y(t)$ is a function belonging to $\h_{n_0,T}\}$.
Clearly, if $Y(\cdot)\in \h_{i_0,T}$, then $Y_t\in K^{\phi,1}\,\forall t\in[0,T]$.
Thus, we can proceed as follows:
\begin{equation}
\begin{aligned}
\PP_{\phi,i_0}\big\{\tau_1&\leq T, \big(X_{(1)}, \alpha_1\big)\notin \K^{\phi,1}\times N_{n_1}\big\}\\
=&\PP_{\phi,i_0}\big(\{\tau_1\leq T, \alpha_1>n_1\}\cup\{\tau_1\leq T,X(\tau_1)\notin \K^{\phi,1}\}\big)\\
=&\PP_{\phi,i_0}\big(\{\lambda_1\leq T, \beta_1>n_1\}\cup\{\lambda_1\leq T,Y(\lambda_1)\notin \K^{\phi,1}\}\big)\\
\leq&\PP_{\phi,i_0}\{\lambda_1\leq T, Y(\cdot)\in \h_{n_0, T},\beta_1>n_1\}+\PP\{Y(\cdot)\notin \h_{n_0}^T\}\\
\leq&\E_{\phi,i_0}\big[\E\big(\1_{\{Y(\cdot)\in \h_{n_0, T}\}}\1_{\{\lambda_1\leq T,\beta_1>n_1\}}|\F^W_T\big)\big]+\eps\\
=&\E_{\phi,i_0}\big[\1_{\{Y(\cdot)\in \h_{n_0, T}\}}\E\big(\1_{\{\lambda_1\leq T,\beta_1>n_1\}}|\F^W_T\big)\big]+\eps\\
=&\E_{\phi,i_0}\big[\1_{\{Y(\cdot)\in \h_{n_0, T}\}}\int_0^T\sum_{i>n_1}q_{i_0,i}(Y_t)\exp\big(-\int_0^tq_{i_0}(Y_s)ds\big)dt\big]+\eps\\
\leq&\E_{\phi,i_0}\big[\1_{\{Y(\cdot)\in \h_{n_0, T}\}}\int_0^T\eps dt\big]+\eps\leq (T+1)\eps.
\end{aligned}
\end{equation}
Similarly, if $(\phi_1, i_1)\in \K^{\phi, 1}\times N_1$, then
$\PP_{\phi_1,i_1}\big\{\tau_1\leq T, \big(X_{(1)}, \alpha_1\big)\notin \K^{\phi,2}\times N_{n_2}\big\}\leq (T+1)\eps.$
 Using  the strong Markov property of $(X_t,\alpha(t))$, we obtain
$$
\begin{aligned}
\PP_{\phi,i_0}\Big\{&\tau_1<T, \big(X_{(1)},\alpha_1\big)\in \K^{\phi,1}\times N_{n_1}, \tau_2\leq T, \big(X_{(2)}, \alpha_2\big)\notin \K^{\phi,2}\times N_{n_2}\Big\}\\
\leq&\PP_{\phi,i_0}\big\{\tau_1<T, \big(X_{(1)},\alpha_1\big)\in \K^{\phi,1}\times N_{n_1}\big\}\\
&\!\!\!\!\times\PP_{\phi,i_0}\Big[\big\{\tau_2\leq T+\tau_1, \big(X_{(2)}, \alpha_2\big)\notin \K^{\phi,2}\times N_{n_2}\big\}\Big|\tau_1<T, \big(X_{(1)},\alpha_1)\big)\in \K^{\phi,1}\times N_{n_1}\!\Big]\\
\leq&\sup\limits_{(\phi_1,i_1)\in \K^\phi_1\times N_{n_1}}\PP_{\phi_1,i_1}\big\{\tau_1\leq T, \big(X_{(1)}, \alpha_1\big)\notin \K^{\phi,2}\times N_{n_2}\big\}\leq (T+1)\eps.
\end{aligned}
$$
Continuing this way, we can show for any $k=1,\dots,m$ that
\begin{equation}\label{e3.16}
\PP_{\phi,i_0}\Big\{\tau_k\leq T, \big(X_{\tau_k},\alpha_k\big)\notin \K^\phi_k\times N_{n_k}, \big(X_{(j)}, \alpha_j\big)\in \K^\phi_j\times N_{n_j},j=1,\dots,k-1\Big\}
\leq (T+1)\eps .
\end{equation}
Consequently,
$$\PP_{\phi,i_0}\Big\{\exists k=1,\dots,m: \tau_k\leq T \text{ and }\big(X_{(k)},\alpha_k\big)\notin \K^\phi_k\times N_{n_k} \Big\}\leq (T+1)m\eps.
$$
Hence, if we choose $\eps=\dfrac{1}{2m(T+1)}\Delta$,
\begin{equation}\label{e3.17}
\begin{aligned}
\PP_{\phi,i_0}\Big\{&\forall k=1,\dots,m: \tau_k> T\text{ or }\alpha_k\in N_{n_k} \Big\}\\
\geq& \PP\Big\{\forall k=1,\dots,m: \tau_k> T\text{ or }\big(X_{(k)},\alpha_k\big)\in \K^\phi_k\times N_{n_k} \Big\}\geq
1-\dfrac\Delta2.
\end{aligned}
\end{equation}
It follows from \eqref{e3.12} and \eqref{e3.17} that
\begin{equation*}
\PP_{\phi,i_0}\Big(\{\tau_{m+1}> T\}\cap\Big\{\forall k=1,\dots,m: \tau_k> T\text{ or }\alpha_k\in N_{n_k} \Big\}\Big)\geq 1-\Delta.
\end{equation*}
It is easily  verified that if $\omega\in \{\tau_{m+1}> T\}\cap\{\forall k=1,\dots,m: \tau_k> T\text{ or }\alpha_k\in N_{n_k} \}$,
then $\alpha(t)\in N_{n_m},\forall t\in[0,T]$.
The assertion of the lemma is proved.
\end{proof}

\end{document}